\documentclass[twoside]{article}
\usepackage{hyperref}

\usepackage{amsmath,amssymb,amsthm}
\renewcommand{\b}[1]{\ensuremath{\mathbf{#1}}} 
\newcommand{\gb}[1]{\ensuremath{\boldsymbol{#1}}} 
\DeclareMathOperator*{\essinf}{ess\,inf}

\usepackage[squaren,gray,pstricks]{SIunits}

\theoremstyle{plain}
\newtheorem{theorem}{Theorem}
\newtheorem{defin}{Definition}
\newtheorem{prop}{Proposition}
\newtheorem{lem}{Lemma}
\newtheorem{cor}{Corollary}

\theoremstyle{remark} 
\newtheorem{remark}{Remark}

\DeclareMathOperator{\tr}{tr}
\DeclareMathOperator{\rank}{rank}
\DeclareMathOperator{\ran}{ran}

\usepackage{enumitem} 

\usepackage{fancyhdr}
\usepackage{color}

\usepackage{graphicx,placeins,subfigure}

\begin{document} 

	\title{Stability in the linearized problem \\ of quantitative elastography}
	\author{Thomas Widlak${}^1$, Otmar Scherzer${}^{1,2}$ 	\\
	\small
	${}^1$ Computational Science Center, University of Vienna, \\ 
	\small Oskar-Morgenstern-Platz~1, A-1090 Vienna, Austria\\ 
	\small
	${}^2$ RICAM, Austrian Academy of Sciences, \\
	\small Altenberger Str.~69, A-4040 Linz, Austria}
	
	\date{May 20th, 2014}

	\pagestyle{empty}
	\maketitle
	\pagestyle{fancy}

	\fancyhf{}
	\renewcommand{\headrulewidth}{0pt}
	
	\fancyhead[LE,RO]{\thepage}
	\fancyhead[CE]{\uppercase{T. Widlak and O. Scherzer}}
	\fancyhead[CO]{\uppercase{Stability in linearized elastography}}

\begin{abstract}
The goal of quantitative elastography is to identify biomechanical parameters from interior displacement data, which are provided by other modalities, such as ultrasound or magnetic resonance imaging. In this paper, we analyze the stability of several linearized problems in quantitative elastography. Our method is based on the theory of redundant systems of linear partial differential equations. We analyze the ellipticity properties of the corresponding PDE systems augmented with the interior displacement data; we explicitly characterize the kernel of the forward operators and show injectivity for particular linearizations. Stability criteria can then be deduced. Our results show stability of shear modulus, pressure and density; they indicate that singular strain fields should be avoided, and show how additional measurements can help in ensuring stability.
\end{abstract}

\section{Introduction}

Elastography is a medical imaging technology; its current applications range from detection of cancer in the breast and in the prostate, liver cirrhosis and characterization of artherosclerotic plaque in hardened coronary vessels \cite{Doy12,ParDoyRub11,WojFarWebThoFis10, AigPalSchoLebJun11, WanChanHunEngTun09,WooRomLerPanBro06,BisPatParCicRic10}. 

Elastography is based on the fact that tissue has high contrast in biomechanical quantities and the health state of organs is reflected in the elastic properties of tissue \cite{KroWheKalGarHal98,AglSko00}. The most important of these is the shear modulus $\mu$, which is the  dominant factor in the propagation of shear waves in tissue; shear wave speed in tissue can change up to a factor of 4 with disease \cite{SarSkoEmeFowPip95}.

Elastography is performed by coupling with various established imaging techniques, such as ultrasound \cite{LerParHolGraWaa88,OphCesPonYazLi91}, MRI \cite{MutLomRosGreMan95,ManOliDreMahKru01} or OCT \cite{SunStaYan11,NahBauRouBoc13} -- What is common to these elastography techniques is that they provide \emph{interior data} of the displacement $\b u|_\Omega$ of the tissue on the imaging domain $\Omega$. According to the specific excitation used, $\b u$ can be space or space-time-dependent (both cases are treated in this work).

In some applications, knowledge of the displacement $\b u$ already gives qualitative diagnostic information (see, e.g., \cite{ZhaBroMosVinMar09} for a dermatological application). More accurate information is provided by quantitative information of the underlying biomechanical parameters. For this one needs to formulate the elasticity problem as a model; for elastography, various models based on linear elasticity, viscoelasticity or hyperelasticity have been considered \cite{Doy12}.

To recover the material parameters, an inverse problem based on an elasticity model has to be solved. Given the displacement $\b u$, the mathematical problems in \emph{quantitative elastography} are the recovery of parameters such as the Lam\'e parameter $\lambda$, the shear modulus $\mu$, the density $\rho$, or recovery of the shear wave speed $\sqrt{\frac{\mu}{\rho}}$ \cite{Doy12,BarBam02}. These problems are non-linear inverse problems.

Much effort in elastography research has been concerned with developing adequate numerical inversion schemes for all kind of experimental varieties of elastography (see, e.g., \cite{McLRen06a,McLRen06,McLRenParWu07,AmmGarJou10,GokBarObe08} and the reviews in \cite{Doy12,SetGoeDhoBruSlo11}); among the proposed algorithms, optimization procedures, which need the linearized inverse problem, form an important class. 

Mathematical results about uniqueness and stability of the inverse problems in elastography have been gathered recently, mostly for the simplest models in linear elasticity. In \cite{JiMcl04,McLYoo04}, it was proved, that -- given one piece of dynamic interior information $\b u$ subject to restrictions such as $\b u\neq 0$ --, one can uniquely recover material parameters. Other uniqueness results have been reported in \cite{BarGok04} for two and more measurements, and in \cite{RodObeBar12} for a hyperelastic model. The recent paper \cite{BalBelImpMon13}, showed unique reconstruction of $\lambda$ and $\mu$ given two sets of exact measurements subject to some non-trivial conditions on the displacements. The stability of the nonlinear problem has been studied in \cite{BalBelImpMon13} using ODE-based and variational tools.

Elastography can be seen as part of \emph{coupled-physics} imaging methods \cite{ArrSch12}. This body of literature centers on novel imaging methods involving more than one physical modality; emphasis is laid upon the quantitative imaging problems for mechanical, optical or electrical parameters that have high diagnostic contrast. For reviews on the typical problems in coupled-physics imaging, see \cite{Amm08,Bal12,Kuc12,WidSch12,ArrSch12,KucKun08}.

In many coupled-physics problems, high resolution in the reconstructions was observed. To explain this phenomenon in a unified manner, a general strategy is to linearize the corresponding nonlinear quantitative problems. Then the stability properties of the linearized problem have been treated with tools of linear PDE theory \cite{KucSte12,Bal12b,BalMos13}. In \cite{KucSte12}, problems in conductivity imaging and quantitative photoacoustic imaging were treated. The linearized forward operators were studied with pseudodifferential theory and shown to be Fredholm, i.e., with stable inversion up to a finite-dimensional kernel.  In \cite{Bal12b}, a general framework was proposed for treating linearizations of quantitative problems with interior data; in this method, the parameters are considered as variables and the information is recast into a single redundant PDE system, to which  theory such as \cite{Sol73} is applied. This method was applied to the power density problem in \cite{Bal12b} and to the problem of 
acousto-optic imaging in  \cite{BalMos13}.

In this article, we treat the linearized problem in quantitative elastography using the general coupled-physics approach in \cite{Bal12b}, using \cite{Sol73}. This is the first time that this technique is applied to a problem in elasticity imaging. We treat stability and explicitly characterize the finite-dimensional kernel, and show injectivity in several versions of the linearized inverse problems.

The structure of the paper is as follows: in Section~\ref{sec:Model}, we review the elasticity model which we are using, and in Section~\ref{sec:SolTheo}, we review necessary background of linear PDE theory. In Section~\ref{sec:Stab}, we apply this theory to the elasticity equation. We investigate how several kinds of linearizations perform analytically, derive the stability results and the characterization of the kernel. In Section~\ref{sec:Disc}, we discuss the stability conditions with respect to the literature. The appendix contains topological lemmas needed for the investigation of the kernel.

\section{Modelling quantitative elastography}
\label{sec:Model}
\subsection{Experiments and interior information in elastography}
\label{subsec:Exp}
The general principle in elastography \cite{Doy12} is to
\begin{itemize}
	\item perturb the tissue using a suitable mechanical source
	\item determine the internal tissue displacement using an ultrasound, magnetic resonance or optical displacement estimation method 
	\item infer the mechanical properties from the interior information, using a mechanical model
\end{itemize}

Note that in elastography, there are two forms of excitation: an elastic deformation from the mechanical source, and the excitation from the ground modality. Also, the reconstruction procedure involves two steps: the recovery of the mechanical displacement $\b u(\b x)|_\Omega$ resp. $\b u(\b x,t)|_\Omega$ from measurements on the boundary, and the recovery of the mechanical parameters and properties from $\b u|_\Omega$. In this last step of quantitative reconstruction (which is treated in this article), the mechanical displacement $\b u|_\Omega$ is also referred to as \emph{interior information}.

\subsection{A linear elasticity model for inhomogeneous linear isotropic media}
There exist different variants of elastography using quasi-static, transient or time-harmonic mechanical excitations, but they can be described by common PDE models \cite{ParDoyRub11,Doy12}.

The elasticity models can be deduced from the equation of motion \cite{Lov44}
\begin{equation}
	\label{eq:EqMot}
	\nabla\cdot\sigma-\rho\b u_{tt}=\b F
\end{equation}
where $\b F(\b x,t)$ is the excitation force density in $\newton\per\meter\cubed$. (We use the convention that letters printed in bold denote \emph{vectors} in $\mathbb{R}^3$.) The mechanical displacement $\b u(\b x,t)$ of the material point which is at position $\b x$ at time $t$ is measured in $[\b u]=\meter$. $\rho(\b x)$ is the density in $\kilogram\per\meter\cubed$. $\sigma=(\sigma_{ij}(\b x,t))_{ij}$ is the mechanical (or Cauchy) stress tensor with unit $\newton\per\meter\squared$.
Here, the divergence of a second-rank tensor $A$ is computed column-wise:
\[
	\nabla \cdot A=\nabla\cdot (\b a_1,\ldots ,\b a_n):=(\nabla \cdot \b a_1,\ldots,\nabla\cdot\b a_n).
\]

The constitutive equation in elasticity (also termed Hooke's law) is
\begin{equation}
\label{eq:Hooke}
	\sigma = C * \varepsilon
\end{equation}
with the stress tensor $\sigma = \sigma_{ij}(\b x,t)$ and the dimensionless strain tensor $\varepsilon = \varepsilon_{kl}(\b x,t)$. Here, $*$ denotes the tensor multiplication. The material properties are incorporated into $C= C_{ijkl}(\b x)$, which is the rank-four stiffness tensor (with unit $\newton\per\meter\squared$).

In most practical settings of elastography, one makes simplifications concerning the material parameters of the biological tissue \cite{Doy12,ParDoyRub11}: Typically, the assumptions are that the material response is \emph{isotropic} and \emph{linear}.

\emph{Isotropy} means that $C$ is reduced to knowledge of two scalar quantites $\lambda$ and $\mu$, such that one has
\begin{equation}
	\label{eq:Isotrop}
	\sigma = 2\mu \varepsilon+\lambda \tr\varepsilon\; \b{Id}, 
\end{equation}
instead of \eqref{eq:Hooke}. Here, $\lambda(\b x)$ is called the first Lam\'e parameter, and $\mu(\b x)$ is called the shear modulus or the second Lam\'e parameter. The physical units of $\lambda$ and $\mu$ are the same as of $\sigma$ and $C$, i.e. $\unit{}{\newton\per\metre\squared}$. 

\emph{Linear material behavior} is encoded in the following representation of the strain~$\varepsilon$:
\begin{equation}
\label{eq:Strain}
	\varepsilon=\frac{1}{2}(\nabla \b u+\nabla \b u^\top).
\end{equation}
Note that with this, the quantity $\tr\varepsilon$ in \eqref{eq:Isotrop} is equal to $\nabla\cdot\b u$.

The equation of motion \eqref{eq:EqMot}, with \eqref{eq:Isotrop} and \eqref{eq:Strain} is then augmented with boundary conditions and appropriate sources. There are different choices for initial and boundary conditions. One option is
\begin{equation}
	\label{eq:Modell}
	\begin{aligned}
		\nabla(\lambda\nabla\cdot\b u)+2\nabla\cdot(\mu\;\varepsilon(\b u))-\rho\b u_{tt} &=\b F\qquad\text{on }\Omega \\
		\b u|_{\partial\Omega}					&= 0 \\
		\b u|_{\Omega\times\{t=0\} }			&= \b g \\
		\partial_t \b u|_{\Omega\times\{t=0\}}	&=\b h.
	\end{aligned}
\end{equation}
Existence, uniqueness and regularity properties for this model follow from the theory in \cite{MclThoYoo11} (for an earlier result for elastodynamic problems, see \cite{DauLio88e}). The initial and boundary values $\b g$ and $\b h$ are prescribed only for the complete the mathematical analysis. In practical physical experiments, the varying excitations enter in the source term $\b F$.

The standard problem of \emph{quantitative elastography} is then to determine the material parameters $\lambda(\b x)$, $\mu(\b x)$ and $\rho(\b x)$ in the equation \eqref{eq:Modell}, given the interior information $\b u(\b x,t)|_\Omega$.

\subsection{Adapting the quantities in the model for inversion}
Models based on \eqref{eq:Isotrop} and \eqref{eq:Modell} are widely used in elastography for simulating the elastic behavior of tissue \cite{ParDoyRub11,Doy12}. Nevertheless, the parameter which one reconstructs, is often only the shear modulus $\mu$. Sometimes, one sets $\lambda\nabla\cdot \b u=0$ or one assumes the incompressibility condition $\nabla\cdot\b u=0$. In these cases, $\lambda$ does not occur in the model at all (see also the discussion in \cite{ManOliDreMahKru01}).

We propose a different definition of quantities for the reconstruction. Precisely, we change the quantities and use the pressure $p$ defined by
\begin{equation}
\label{eq:Pressure}
	p(\b x,t):= \lambda(\b x) \nabla\cdot\b u(\b x,t).
\end{equation}
With these quantities, it follows from \eqref{eq:Modell} that
\begin{equation}
\label{eq:ModModel}
	\nabla p+2\nabla\cdot(\mu\;\varepsilon(\b u))-\rho\b u_{tt} =\b F,
\end{equation}
The inverse problem is now to recover $p(\b x,t)$, as well as $\mu(\b x)$ and $\rho(\b x)$, given $\b u(x,t)|_\Omega$. It is this problem which we address in our stability analysis.

Definition \eqref{eq:Pressure} has been used before, see e.g. \cite{McLZhaMan10,RagYag94}). Note that in tissue, one has that $\nabla\cdot\b u\ll 1$. Because of ill-posedness of differentiation, the quantity $\nabla\cdot \b u$ cannot be computed accurately from the data $\b u$ in experiments. On the other hand, one has that $\lambda\gg 1$. In numerical simulations, the pressure $p$ turns out to be of order 1 and therefore should not be neglected \cite{McLZhaMan10}.

Note that that $p(\b x,t)$ is an elastic quantity, but not a material parameter: it depends on the particular displacement field induced by the excitation. Knowledge of $p$ may or may not prove to be useful for diagnostic purposes. The reason, though, for introducing this quantity in the inversion model is that it numerically turned out to be useful. It was numerically observed that keeping $p$ in the model improves the reconstruction of the shear modulus $\mu$ \cite{McLZhaMan10}.

In our analysis, we will give a mathematical reason for using \eqref{eq:ModModel} instead of the first equation in \ref{eq:Modell}. Before we come to that, we give the relevant background from PDE theory which we use in our work.

\section{A result from linear PDE theory}
\label{sec:SolTheo}
We first treat the background from the general theory of linear PDE systems.

Let $\Omega$ be a bounded domain in $\mathbb{R}^n$ (smoothness requirements on $\Omega$ are specified later). We consider the redundant system of linear partial differential equations 
\begin{equation}
\label{eq:PDEred}
	\begin{aligned}
			\mathcal L(\b x,\frac{\partial}{\partial \b x})\b u &= \mathcal S & \text{on } \Omega         \\
			\mathcal B(\b x,\frac{\partial}{\partial \b x})\b u &= \varphi    & \text{on } \partial\Omega
	\end{aligned}
\end{equation}
for $m$ unknown functions $\b u(\b x)=(u_1(\b x),\ldots\,u_m(\b x)$), comprising in total $M$ equations. Here, $\mathcal L(\b x,\frac{\partial}{\partial \b x})$ is a matrix differential operator of dimension $M\times m$,

\begin{equation}
\label{eq:OperatorL}
	\mathcal L(\b x,\frac{\partial}{\partial \b x})=	
	\begin{pmatrix}
		L_{11}(\b x,\frac{\partial}{\partial\b x}) &\ldots& L_{1m}(\b x,\frac{\partial}{\partial\b x}) \\
		\vdots & \vdots & \vdots \\
		L_{M1}(\b x,\frac{\partial}{\partial\b x}) &\ldots& L_{Mm}(\b x,\frac{\partial}{\partial\b x})
	\end{pmatrix}.
\end{equation}

For each $1\leq i\leq M$, $1\leq j\leq m$ and for each point $\b x$, $L_{ij}(\b x,\frac{\partial}{\partial\b x})$ is a polynomial in $\frac{\partial}{\partial\b x}=(\frac{\partial}{\partial x_1}, \ldots, \frac{\partial}{\partial  x_n})$. Redundancy of the system means that there are possibly more equations than unknowns: $M\geq m$.

Similarly, $\mathcal B(\b x,\frac{\partial}{\partial\b x})$ has entries $B_{kj}(\b x,\frac{\partial}{\partial\b x})$ for $1\leq k\leq Q, 1\leq j \leq m$, consisting of $Q$ equations at the boundary. The operations are again polynomial in the second variable. -- $\mathcal S(\b x)$ is a vector of length $M$, and $\varphi(\b x)$ is a vector of length $Q$.

We now define the notions of \emph{ellipticity} and the \emph{principal part} of $\mathcal L$ and $\mathcal B$, respectively, in the sense of Douglis and Nirenberg \cite{DouNir55}.

\begin{defin}
	\label{def:Ell}
	Let integers $s_i,t_j\in\mathbb{Z}$ be given for each row $1\leq i\leq M$ and column $1\leq j\leq m$ with the following property: For $s_i + t_j\geq 0$, the order of $L_{ij}$ does not exceed $s_i + t_j$. For $s_i + t_j<0$, one has $L_{ij}=0$. Furthermore, the numbers are normalized so that for all $i$ one has $s_i\leq 0$. Such numbers $s_i,t_j$ are called \emph{Douglis-Nirenberg numbers}.
	
	The \emph{principal part} of $\mathcal L$ for this choice of numbers $s_i, t_j$ is defined as the matrix operator $\mathcal L_0$ whose entries $L_{0,ij}$ are composed of those terms in $L_{ij}$ which are exactly of order $s_i+t_j$.
	
	The principal part $\mathcal B_0$ of $\mathcal B$ is composed of the entries $B_{0,ij}$, which are composed of those terms in $B_{kj}$ which are exactly of order $\sigma_k + t_j$. The numbers $\sigma_k, 1\leq k \leq Q$ are computed as 
\begin{equation}
	\label{eq:Sigma}
	\sigma_k:=\max_{1\leq j\leq m} (b_{kj}-t_j),
\end{equation}
where $b_{kj}$ denotes the order of $B_{kj}$.
	
	Real directions $\gb \xi\neq 0$ with $\rank\mathcal L_0(\b x,\i\gb\xi) < m$ are called \emph{characteristic directions} of $\mathcal L$ at $\b x$. (The complex unit is denoted by the symbol $\i=\sqrt{-1}$.) The operator $\mathcal L(\b x,\frac{\partial}{\partial\b x})$ is said to be \emph{overdetermined elliptic} in $\Omega$ if for all $\b x\in\overline\Omega$ and for all real vectors $\gb \xi\neq 0$ one has that
\begin{equation}
\label{eq:OvDetEll}
	\rank\mathcal L_0(\b x,\i\gb \xi) = m
\end{equation}
for the $M\times m$ matrix $\mathcal L_0(\b x,\i\gb \xi)$.
\end{defin}

To illustrate in an example, we can consider the system
\begin{equation}
\label{eq:LopatBsp1}
	\begin{pmatrix} 
		 u_1 &+& \Delta u_2 \\
		\frac{\partial}{\partial{x_1}} u_1\\ \frac{\partial}{\partial{x_2}} u_1\end{pmatrix}=\begin{pmatrix} 0 \\ 
	 	f	 \\
	 	g 
	\end{pmatrix},\qquad\qquad 
	\begin{pmatrix}
		u_1|_{\partial\Omega}		\\
		& \nabla u_2\cdot\gb \nu|_{\partial\Omega}
	\end{pmatrix}=
	\begin{pmatrix}
		h\\
		k
	\end{pmatrix},
\end{equation}
where $\Omega$ is the unit circle in $\mathbb{R}^2$. With $\gb \nu$, we denote the unit normal on $\partial\Omega$. If we choose numbers $(t_j)_{j=1}^2= (1,3), (s_i)_{i=1}^3=(-1,0,0)$, we have the principal symbols
\begin{equation}
\label{eq:LopatBsp3}
	\mathcal L_0(\i\gb \xi) = 
	\begin{pmatrix} 
	1& -|\gb\xi|^2 \\ 
	 \i\xi_1 & 0\\ 
	 \i\xi_2 &0 
	\end{pmatrix},\qquad
	\mathcal B_0(\i\gb \xi)=
	\begin{pmatrix} 
	 1 & 0 \\  
	0 & 0 
	\end{pmatrix}
\end{equation}
and  $(\sigma_i)_{i=1}^2=(-1,-1)$. -- If we choose Douglis Nirenberg numbers $(t_j)_{j=1}^2= (1,2), (s_i)_{i=1}^3=(0,0,0)$, we have  the principal symbols
\begin{equation}
\label{eq:LopatBsp2}
	\mathcal L_0(\i\gb \xi) = 
	\begin{pmatrix} 
		0& -|\gb\xi|^2 \\
	 	\i\xi_1 &0\\
	 	\i\xi_2 &0
	 	\end{pmatrix},\qquad
	 \mathcal B_0(\i\gb \xi)=
	 \begin{pmatrix} 
			 	1 &0\\
	 	0&\i\gb\xi\cdot\gb \nu
	 \end{pmatrix}
\end{equation}
and $(\sigma_i)_{i=1}^2=(-1,-1)$. 

Note that the principal symbols differ in this case. Nevertheless, with both choices of numbers, $\mathcal L$ is overdetermined elliptic, as there exists a non-vanishing subdeterminant of $\mathcal L_0$ in both cases.

Next we define the condition of $\mathcal B$ covering $\mathcal L$, or the \emph{Lopatinskii boundary condition} \cite{Sol73}.
\begin{defin}
	\label{def:Lopat}
	Fix $\b y\in\partial\Omega$, and let $\gb \nu$ be the inward unit normal vector at $\b y$. Let $\gb \zeta$ be any non-zero tangential vector to $\Omega$ at $\b y$. Consider the half-line $\{\b y+z\,\gb \nu,z>0\}$ and the following system of ordinary differential equations on it:
	\begin{equation}
	\label{eq:Lopat}
		\begin{aligned}
				\mathcal{L}_0(\b y, \i\gb \zeta + \gb \nu\frac{d}{dz})\tilde{\b u}(z)&=0\qquad\qquad z>0\\
				\mathcal{B}_0(\b y, \i\gb \zeta + \gb \nu\frac{d}{dz})\tilde{\b u}(z)&=0\qquad\qquad z=0. 
		\end{aligned}
	\end{equation}
	Consider the vector space of all solutions $\tilde{\b u}$ of \eqref{eq:Lopat} which satisfy $\tilde{\b u}(z)\to 0$ for $z~\to~\infty$. If this vector space consists just of the trivial solution $\tilde{\b u}(z)\equiv 0$, then the Lopatinskii condition is said to be fulfilled for the pair $(\mathcal L,\mathcal B)$ at $\b y$, or $\mathcal B$ \emph{covers} the operator $\mathcal L$ at $\b y$.
\end{defin}
In the example \eqref{eq:LopatBsp1} above, the equations \eqref{eq:Lopat}, together with the orthogonality condition $\gb\nu\cdot\gb\zeta=0$, yield 
\[
\begin{aligned}
	(-|\gb \zeta|^2+\frac{d^2}{dz^2})\tilde u_2(z) &= 0 \\
	(\i\zeta_1+\nu_1\frac{d}{dz})\tilde u_1(z) &=0 \\
	(\i\zeta_2+\nu_2\frac{d}{dz})\tilde u_1(z) &=0.	
\end{aligned}
\]
The last two equations, together with $\gb\zeta\cdot\gb\nu=0$, imply $\tilde u_1(z)=0$. For $\tilde u_2(z)$, there is the solution $\tilde u_2(z)=C e^{-| \gb\zeta|z}$ going to $0$ for $z\to\infty$. -- In the case \eqref{eq:LopatBsp3}, there is no requirement to restrain $u_2(z)$, therefore the Lopatinskii boundary condition is not satisfied with this choice of Douglis-Nirenberg numbers. In the other case \eqref{eq:LopatBsp2}, the requirement $\mathcal B_0(\b y,\i\gb\zeta + \gb\nu\frac{d}{dz})\tilde{\b u}(z)=0$ leads to $\frac{d}{dz}\tilde u_2(z)=-C|\gb\zeta|e^{-|\gb\zeta| z}=0$, therefore $C=0$ and consequently $\tilde u_2(z)=0$. Therefore the Lopatinskii boundary condition is satisfied in this case. 

A typical example of an overdetermined elliptic systems with Lopatinskii boundary conditions is
\begin{equation}
\label{eq:Helmholtz}
	\begin{aligned}
		\nabla\times\b u&= f\\
	 	\nabla\cdot\b u &= g
	 \end{aligned}
\end{equation}
on a domain $\Omega$ with normal component $\b u\cdot\b v|_{\partial\Omega}$ given on the boundary.

Another example of an overdetermined elliptic system with Lopatinskii boundary condition is the system of time-harmonic Maxwell's equations, where $\b u=(\b H, \b E)$ satisfies
\begin{equation}
\label{eq:Maxwell}
	\begin{aligned}
		\nabla\times\b H &= -\kappa_1 \b E \qquad & \nabla \times \b E&= \kappa_2\b H\\
		\nabla\cdot \b H &=0 & \qquad \nabla\cdot \b E &= \rho
	\end{aligned}
\end{equation}
with the normal component of $\b H$ as well as the tangential component of $\b E$ given on the boundary (see \cite[\S2]{Sol73} for \eqref{eq:Helmholtz} and \eqref{eq:Maxwell})

In the context of hybrid imaging, examples of overdetermined elliptic systems with Lopatinskii boundary conditions have been considered in \cite{Bal12b,BalMos13}.

\medskip

For investigating the stability for linearized quantitative elastography, we are going to use the a-priori estimate in \cite{Sol73} for the solutions of system \eqref{eq:PDEred}. This theory does not need smooth coefficients, but coefficients in the Sobolev spaces $W^\alpha_p(\Omega)$ (for the usual definition, also for noninteger values of $\alpha$, see \cite{Ada75}).
In the setting of \cite{Sol73} with Douglis-Nirenberg numbers $t_j,s_i,\sigma_k$, one has that the operator $\mathcal A$ with
\begin{equation}
\label{eq:GesSys}
	\mathcal A\b u=\begin{pmatrix}\mathcal L \b u\\ \mathcal B\b u\end{pmatrix}
\end{equation}
acts on the space
\begin{equation}
\label{eq:Dpl}
	D(p,l) := W^{l+t_1}_p(\Omega)\times \ldots\times W^{l+t_m}_p(\Omega),
\end{equation}
where $l\geq 0$, $p>1$.
Under suitable restrictions on the coefficients $L_{ij}$ and $B_{kj}$ (specified below in the conditions of the theorem), the operator $\mathcal A$ is bounded with range in 
\begin{equation}
\label{eq:Rpl}
	R(p,l):=W^{l-s_1}_p(\Omega)\times \ldots\times W^{l-s_M}_p(\Omega)\times W^{l-\sigma_1-\frac{1}{p}}_p(\partial\Omega)\times \ldots\times W^{l-\sigma_Q-\frac{1}{p}}_p(\partial\Omega).
\end{equation}

Using the operator $\mathcal A$ in \eqref{eq:GesSys}, the equations \eqref{eq:PDEred} read
\begin{equation}
\label{eq:Gleichung}
\mathcal A\b u=\begin{pmatrix}\mathcal S\\ \varphi\end{pmatrix}.
\end{equation}

In formulating the restrictions on the coefficients of $\mathcal L$ and $\mathcal B$, we simplify the version of \cite[Thm. 1.1]{Sol73} for the following result.

\begin{theorem}
	\label{thm:Sol}
	Let integers $l\geq 0, p>1$ be given. Let $(\mathcal S,\varphi)$, the data from \eqref{eq:PDEred} be in $R(p,l)$ as defined in \eqref{eq:Rpl}. Let Douglis-Nirenberg numbers $s_i$ and $t_j$ be given for $\mathcal L$ in \eqref{eq:PDEred}, and let $\sigma_k$ be as in Definition \ref{def:Ell}. Let $\Omega$ be a bounded domain with boundary in $C^{l+\max t_j}$. Assume furthermore that $p(l-s_i)>n$ and $p(l-\sigma_k)>n$ for all $i$ and $k$. Let the coefficients of $L_{ij}$ be in $W^{l-s_i}_p(\Omega)$, and let the coefficients of $B_{kj}$ be in $W^{l-\sigma_k-\frac{1}{p}}(\Omega)$. Then the following statements are equivalent:
	\begin{enumerate}
	\item $\mathcal L$ in \eqref{eq:PDEred} is overdetermined elliptic (see \eqref{eq:OvDetEll}) and the Lopatinskii covering condition \eqref{eq:Lopat} is fulfilled for ($\mathcal L,\mathcal B$) on $\partial\Omega$. 
	
	\item There exists a left regularizer $\mathcal R$ for the operator $\mathcal A=\mathcal L\times\mathcal B$ in \eqref{eq:GesSys}, that is, we have
	\begin{equation}
	\label{eq:leftReg}
		\mathcal R\mathcal A=\mathcal I-\mathcal T
	\end{equation}
	with $\mathcal T$ compact from $R(p,l)$ in \eqref{eq:Rpl} to $D(p,l)$ in \eqref{eq:Dpl}.
	
	\item The following a-priori estimate holds
	\begin{equation}
	\label{eq:SolloEst}
		\sum_{j=1}^m  \| u_j \|_{W_p^{l+t_j}(\Omega)} \leq
		C_1 (\sum_{i=1}^M\| \mathcal S_i\|_{W_p^{l-s_i}(\Omega)} + \sum_{k=1}^Q \| \varphi_k \|_{W_p^{l-\sigma_j-\frac{1}{p}}(\partial\Omega)}) +
		C_2 \sum_{t_j>0} \| u_j\|_{L^p(\Omega)},
	\end{equation}
where $u_j$ is the $j$-th component of the solution $\b u$ of \eqref{eq:GesSys}
	\end{enumerate}
\end{theorem}
The assertion of the theorem gives a criterion for the existence of a left regularizer for the overdetermined redundant systems. For the case of boundary value problems for square systems with $M=m$, such an equivalence is established in the classical work of \cite{Agr65}. For square systems, one has the stronger statement that ellipticity and Lopatinskii condition are equivalent to the Fredholm property of a differential operator (which also needs the existence of a right regularizer). -- The criterion for redundant systems with $M\geq m$ was established in \cite{Sol73}, and investigates the stability estimate, and even gives a representation formula for the solution, provided it exists. The existence of a right regularizer $\mathcal Q$ with $\mathcal A\mathcal Q = \mathcal I - \mathcal T$ (which would yield local existence) cannot be assured in general for overdetermined systems. 

We will exploit this criterion for the linearized version of quantitative elastography.

\section{Stability analysis}
\label{sec:Stab}
\subsection{Setting and notation}
\label{sec:Setting}
For treating the hybrid imaging problem described in \ref{subsec:Exp}, we take the adapted forward model \eqref{eq:Modell} + \eqref{eq:Pressure}. We recast the equations of the forward problem for the displacement, and the interior information from the measurements, respectively, into a single system of partial differential equations:
\begin{equation}
\label{eq:GesModell}
	\begin{aligned}
		\nabla p+2\nabla\cdot(\mu\;\varepsilon(\b u))-\rho\b u_{tt} &=\b F \\
		\b u(x,t)  								&= \b K \\
		\b u|_{\partial\Omega}					&= 0 \\
		\b u|_{\Omega\times\{t=0\} }			&= \b g \\
		\partial_t \b u|_{\Omega\times\{t=0\}}	&=\b h 
	\end{aligned}
\end{equation}
Here, we are given $\b F$ (the excitation force), $\b g$ and $\b h$, as well as the interior information $\b K$. We formally keep $\b u$ and $\b K$ distinct in the second equation, as $\b u$ is treated as variable of the system, and $\b K$ represents the data. -- We aim at a quantitative estimate such as \eqref{eq:SolloEst} with the measurement data in the inhomogeneity.

We consider the system \eqref{eq:GesModell}, for the variables $\b u$, $p$, $\mu$, $\rho$; therefore, we consider it as nonlinear, involving multiplication of the unknowns. In order to make it tractable for the analysis, we linearize this system to provide equations of the form \eqref{eq:PDEred}. Several of these linearizations will be considered in the subsequent theory.

We first consider the parameter-to-solution operator
\begin{equation}
\label{eq:RefDispl}
	\mathcal V:(p(\b x,t),\mu(\b x),\rho(\b x))\qquad\mapsto\qquad \b u(\b x,t),
\end{equation}
which maps the choice of parameters $(p,\mu,\rho)$ to the displacement field $\b u$ satisfying \eqref{eq:GesModell}. Then we consider the linearization of $\mathcal V$ at a reference state $(p,\mu,\rho)$,
\[
	\mathcal V'(p,\mu,\rho):\begin{pmatrix}\delta p\\ \delta \mu \\ \delta\rho\end{pmatrix} \qquad\longmapsto\qquad \begin{pmatrix} \delta u_1 \\ \delta u_2 \\ \delta u_3\end{pmatrix}=:\delta \b u.
\]
By formal differentiation of \eqref{eq:GesModell}, we find that, at the reference state $\b u=\mathcal V(p,\mu,\rho)$, the increment $\delta \b u$ satisfies the equations 
\begin{equation}
\label{eq:LinEqns}
\begin{aligned}
	\nabla \delta p + 2\nabla\cdot(\delta\mu\;\varepsilon(\b u))+ 2\nabla\cdot(\mu\;\varepsilon(\delta\b u))-\delta\rho\;\b u_{tt}-\rho(\delta\b u)_{tt} &= 0\\
	\delta \b u &= \delta \b K \\
		\delta \b u|_{\partial\Omega}				& = 0	 \\
		\delta \b u|_{\Omega\times\{t=0\} }			& = 0	 \\
		\partial_t \delta \b u|_{\Omega\times\{t=0\}}	& = 0 .
\end{aligned}
\end{equation}
Note that $\b F$ does not depend on the reference state, therefore no inhomogeneity appears in the first equation in \eqref{eq:LinEqns}.

Observe that \eqref{eq:LinEqns} is a system of differential equations for the functions $(\delta p,\delta\mu,\delta\rho,\delta u_1,\delta u_2,\delta u_3)=(\delta p,\delta\mu,\delta\rho,\delta\b u)$. The system is linear in these unknowns.

We write
\begin{equation}
\label{eq:FPMuRho}
	\mathcal F_{p\mu\rho}(\delta p,\delta\mu,\delta\rho,\delta \b u):= \nabla \delta p + 2\nabla\cdot(\delta\mu\;\varepsilon(\b u))+ 2\nabla\cdot(\mu\;\varepsilon(\delta\b u))-\delta\rho\;\b u_{tt}-\rho(\delta\b u)_{tt},
\end{equation}
\begin{equation}
\label{eq:Boundary}	
		\mathcal B(\delta \b u):=\begin{pmatrix}
		\delta \b u|_{\partial\Omega}					 \\
		\delta \b u|_{\Omega\times\{t=0\} }				 \\
		\partial_t \delta \b u|_{\Omega\times\{t=0\}}	\end{pmatrix},
\end{equation}
and we introduce the operator
\begin{equation}
\label{eq:LPMuRho}
\begin{aligned}
	\mathcal L_{p\mu\rho}(\delta p,\delta\mu,\delta\rho,\delta \b u)&:=
	\begin{pmatrix}
		\mathcal F_{p\mu\rho}(\delta p,\delta\mu,\delta\rho,\delta \b u) \\	
		\delta\b u
	\end{pmatrix}	\\
	&=
	\begin{pmatrix} \nabla \delta p + 2\nabla\cdot(\delta\mu\;\varepsilon(\b u))+ 2\nabla\cdot(\mu\;\varepsilon(\delta\b u))-\delta\rho\;\b u_{tt}-\rho(\delta\b u)_{tt}\\
	\delta\b u\end{pmatrix}.
	\end{aligned}
\end{equation}

The operator $\mathcal L_{p\mu\rho}$ in \eqref{eq:LPMuRho} is the linearization of the redundant system \eqref{eq:GesModell} with respect to $p$, $\mu$ and $\rho$. Note that $\mathcal L_{p\mu\rho}$ is a matrix differential operator like $\mathcal L$ in \eqref{eq:PDEred}.

Apart from $\mathcal L_{p\mu\rho}$, we also introduce the operators corresponding to directional derivatives with respect to only \emph{one} parameter:
\begin{align}
\label{eq:LP}
	\mathcal L_p(\delta p,\delta \b u):=\mathcal L_{p\mu\rho}(\delta p,0,0,\delta \b u)&= \begin{pmatrix}\nabla \delta p + 2\nabla\cdot(\mu\;\varepsilon(\delta\b u))-\rho(\delta\b u)_{tt}  \\  \delta \b u\end{pmatrix}\\
\label{eq:LMu}
	\mathcal L_\mu(\delta \mu,\delta \b u):=\mathcal L_{p\mu\rho}(0,\delta\mu,0,\delta \b u)&=   \begin{pmatrix}2\nabla\cdot(\delta\mu\;\varepsilon(\b u)) + 2\nabla\cdot(\mu\;\varepsilon(\delta\b u))-\rho(\delta\b u)_{tt}\\ \delta \b u\end{pmatrix}\\
\label{eq:LRho}
	\mathcal L_\rho(\delta \rho,\delta \b u):=\mathcal L_{p\mu\rho}(0,0,\delta\rho,\delta \b u)&=  \begin{pmatrix} 2\nabla\cdot(\mu\;\varepsilon(\delta\b u))-\delta\rho\;\b u_{tt}-\rho(\delta\b u)_{tt}  \\\delta \b u\end{pmatrix}
\end{align}
We also use the combination 
\begin{align}
\label{eq:LPMu}
	\mathcal L_{p\mu}(\delta p,\delta \mu,\delta \b u):=\mathcal L_{p\mu\rho}(\delta p,\delta\mu,0,\delta \b u)&=   \begin{pmatrix}\delta p + 2\nabla\cdot(\delta\mu\;\varepsilon(\b u)) + 2\nabla\cdot(\mu\;\varepsilon(\delta\b u))-\rho(\delta\b u)_{tt}\\ \delta \b u\end{pmatrix}.
\end{align}
In the stability analysis, we will make a comparison of the properties of these operators.

In the model \eqref{eq:GesModell} which we started from, we incorporated the definition of the pressure in \eqref{eq:Pressure}. An alternative is to use the original model \eqref{eq:Modell} only, which involved the first Lam\'e parameter $\lambda$. This corresponds to re-substituting $p=\lambda\nabla\cdot\b u$ in \eqref{eq:GesModell}. In exact analogy to constructing $\mathcal L_{p\mu\rho}$ one can form the forward operator $\mathcal V_\lambda$. One then considers its linearization $\mathcal V'_\lambda(\lambda,\mu,\rho):(\delta\lambda,\delta\mu,\delta\rho)\mapsto\delta\b u$, and introduces the operator $\mathcal L_{\lambda\mu\rho}$. We particularly will consider
\begin{equation}
\label{eq:LLambda}
	\mathcal L_\lambda(\delta \lambda,\delta \b u):=\mathcal L_{\lambda\mu\rho}(\delta\lambda,0,0,\delta\b u)=
	\begin{pmatrix} 
	\nabla( \delta \lambda\;\nabla\cdot\b u) + \nabla(\lambda\;\nabla\cdot\delta\b u) + 
	2\nabla\cdot(\mu\;\varepsilon(\delta\b u))-
	\rho(\delta\b u)_{tt}  \\ \deltaÊ\b u
	\end{pmatrix}
\end{equation}

Up to now, we defined several differential operators $\mathcal L$ of form \eqref{eq:OperatorL}. As in \eqref{eq:GesSys}, we now combine them with boundary data and then form equations \[
\mathcal A\b w = (\mathcal S,0)\]
 as in \eqref{eq:PDEred} resp. \eqref{eq:Gleichung}. The vector $\b w$ changes according to the variables which are in the system, specified below. For the inhomogeneity $\mathcal S$, we have 
\begin{equation}
\label{eq:Inhom}
\mathcal S=(0,0,0,\delta K_1,\delta K_2, \delta K_3)^\top.
\end{equation}
We introduce the operator 
\begin{equation}
\label{eq:SystemPMuRho}
	\mathcal A_{p\mu\rho}(\delta p,\delta\mu,\delta\rho,\delta \b u)
	=	
	\begin{pmatrix}
		\mathcal L_{p\mu\rho} \;(\delta p,\delta\mu,\delta\rho,\delta\b u) \\
		 \mathcal B \;(\delta \b u)|_{\partial\Omega}
	\end{pmatrix},
\end{equation}
from which we get the following specializations.

The first system is
\begin{equation}
\label{eq:SystemP}
	\mathcal A_p \;(\delta p,\delta\b u):= 
\mathcal A_{p\mu\rho}(\delta p,0,0,\delta \b u)	
	=
	\begin{pmatrix}
		\mathcal L_p \;(\delta p,\delta\b u) \\
		 \mathcal B \;(\delta\b u)|_{\partial\Omega}
	\end{pmatrix}
	=
	\begin{pmatrix}
		\mathcal S\\ 
		0
	\end{pmatrix}.
\end{equation}
In $\mathcal A_p$, there are more equations than unknowns, namely 6 equations in $\Omega$ for 4 unknowns $(\delta p,\delta u_1,\delta u_2,\delta u_3)$.

The system 
\begin{align}
\label{eq:SystemMu}
	\mathcal A_\mu \;(\delta \mu,\delta\b u):= 
\mathcal A_{p\mu\rho}(0,\delta\mu,0,\delta \b u)	= 
	\begin{pmatrix}
		\mathcal L_\mu \;(\delta \mu,\delta\b u) \\
		 \mathcal B \;(\delta\b u)|_{\partial\Omega}
	\end{pmatrix}
	=
	\begin{pmatrix}
		\mathcal S\\ 
		0
	\end{pmatrix}
\end{align}
has 6 interior equations for the unknowns $(\delta \mu,\delta u_1,\delta u_2,\delta u_3)$.

The system
\begin{align}
\label{eq:SystemRho}
	\mathcal A_\rho \;(\delta \rho,\delta\b u):= 
\mathcal A_{p\mu\rho}(0,0,\delta\rho,\delta \b u)	= 
	\begin{pmatrix}
		\mathcal L_\rho \;(\delta \rho,\delta\b u) \\
		 \mathcal B \;(\delta\b u)|_{\partial\Omega}
	\end{pmatrix}
	=
	\begin{pmatrix}
		\mathcal S\\ 
		0
	\end{pmatrix}
\end{align}
has 6 interior  equations for the unknowns $(\delta \rho,\delta u_1,\delta u_2,\delta u_3)$.

The system
\begin{align}
\label{eq:SystemPMu}
	\mathcal A_{p\mu} \;(\delta p,\delta\mu,\delta\b u):= 
\mathcal A_{p\mu\rho}(\delta p,\delta\mu,0,\delta \b u)	= 
	\begin{pmatrix}
		\mathcal L_{p\mu} \;(\delta p,\delta\mu,\delta\b u) \\
		 \mathcal B \;(\delta\b u)|_{\partial\Omega}
	\end{pmatrix}
	=
	\begin{pmatrix}
		\mathcal S\\ 
		0
	\end{pmatrix}
\end{align}
has 6 interior  equations for the unknowns $(\delta \rho,\delta u_1,\delta u_2,\delta u_3)$.

All of these are linear differential systems which are redundant systems of form \eqref{eq:PDEred} with $M\geq m$. Therefore we can apply the methodology of Section~\ref{sec:SolTheo} to these.

Up to now, we only incorporated information from one imaging experiment in our operators. It is possible, though, to conduct more than one imaging experiment, and the consideration of multiple measurements in the inverse problem is typical in hybrid imaging. Therefore, we can use different excitations $\b F_i$ and possibly different functions $\b g_i,\b h_i$ in \eqref{eq:GesModell} and obtain different versions of $\b u_i(\b x,t)$ and interior information $\b K_i(\b x,t)$. For each experiment, we also have a different variable $p_i$ in the system. While these quantities change with each excitation, the material parameters $\lambda$, $\mu$ and $\rho$ remain the same. 

For example, we write 
\begin{align}
\label{eq:SystemMu2}
	\mathcal A_\mu^{(2)} \;(\delta \mu,\delta\b u_1, \delta\b u_2):= 
	\begin{pmatrix}
		\mathcal L_\mu^{(1)}(\delta \mu,\delta \b u_1) \\
		\mathcal B \;(\delta\b u_1)|_{\partial\Omega} \\
		\mathcal L_\mu^{(2)}(\delta \mu,\delta \b u_2) \\
		\mathcal B \;(\delta\b u_2)|_{\partial\Omega}
	\end{pmatrix}
	=
	\begin{pmatrix}
		\mathcal S^{(1)}\\ 
		0\\ 
		\mathcal S^{(2)}\\
		0
	\end{pmatrix}
\end{align}
for the system corresponding to 2 experiments. Here, the operators $\mathcal L_\mu^{(i)}$ for $i=1,2$ are
\[
\mathcal L_\mu^{(i)}(\delta\mu,\delta\b u_i)
=
\begin{pmatrix}
		2\nabla\cdot(\delta\mu\;\varepsilon(\b u_i)) + 2\nabla\cdot(\mu\;\varepsilon(\delta\b u_i))-\rho(\delta\b u_i)_{tt}\\ \delta \b u_i\end{pmatrix}.
\]
In the inhomogeneity, we have the quantities
$\mathcal S^{(i)}=(0,0,0,\delta K_1^{(i)},\delta K_2^{(i)},\delta K_3^{(i)})$ for $i=1,2$.

Comparison of $\mathcal A_\mu$ in \eqref{eq:SystemMu} with $\mathcal A_\mu^{(2)}$ in \eqref{eq:SystemMu2} shows the effect of adding one more experiment in the system: there are 6 more equations and 3 new variables: together 12 equations and 7 unknowns.

The shown procedure of addition addition experiments can be applied to any of the operators in \eqref{eq:SystemMu}, \eqref{eq:SystemRho}, \eqref{eq:SystemPMuRho}. For each experiment we add, the inequality $M\geq m$ in the nomenclature of Section~\ref{sec:SolTheo} is fulfilled and the system is redundant.

\bigskip 
For reasons which are apparent later, we can augment the boundary operator with additional constraints and use 
\begin{equation}
\label{eq:BoundNew}
	\begin{aligned}
	\mathcal B'(\delta \b u)=\begin{pmatrix}
		\delta \b u|_{\partial\Omega}					 \\
		\delta \b u|_{\Omega\times\{t=0\} }				 \\
		\partial_t \delta \b u|_{\Omega\times\{t=0\}}	\\
		\delta p|_{\partial\Omega}\\
		\delta \mu|_{\partial\Omega}\end{pmatrix} = 0 .
	\end{aligned}
\end{equation}
Normally, we will use \eqref{eq:Boundary} as boundary operator. The conditions \eqref{eq:BoundNew} will be used in special cases which are separately indicated.

\bigskip

Note that in this section, we have given the general form of the linearization operators for the \emph{dynamic case} on a cylindrical domain $\Omega\times T$. Sometimes, we will consider these operators in the \emph{quasi-static case} with 
\begin{equation}
\label{eq:QuasiStatic}
	\b u_{tt} = (\delta \b u)_{tt} = 0,
\end{equation}
using only the spatial domain $\Omega$. This is specially indicated in each case.

\subsection{Ellipticity}
\label{sec:Ell}
We want to apply the methodology of Section~\ref{sec:SolTheo}, and use the criterion in Theorem~\ref{thm:Sol}. Therefore, we have to we determine the ellipticity condition in Definition~\ref{def:Ell} for the operators $\mathcal L_p$, $\mathcal L_\mu$, $\mathcal L_\rho$ in \eqref{eq:LP}, \eqref{eq:LMu}, \eqref{eq:LRho}. We first determine the principal symbol and possible characteristic directions for the operator $\mathcal L_{p\mu\rho}$ in \eqref{eq:LPMuRho}, which is treated in Proposition \ref{prop:Ellip}. From this analysis we then draw some corollaries concerning the ellipticity of $\mathcal L_p$, $\mathcal L_\mu$, $\mathcal L_\rho$, as well as ellipticity of $\mathcal L_{p\mu}$.

For the analysis of $\mathcal L_{p\mu\rho}$, we choose Douglis-Nirenberg numbers $(t_j)_{j=1}^6=(1,1,0,2,2,2)$ and  corresponding to the variables $(\delta p,\delta\mu,\delta\rho,\delta u_1,\delta u_2,\delta u_3)$ and numbers $(s_i)_{i=1}^6=(0,0,0,-2,-2,-2)$ corresponding to the six equations. If there are less variables in the system (as in $\mathcal L_p$, $\mathcal L_\mu$, $\mathcal L_\rho$, $\mathcal L_{p\mu}$), then only the corresponding Douglis-Nirenberg numbers are used (e.g., for the analysis of $\mathcal L_p$, we have the numbers $(1,2,2,2)$ for the variables ($\delta p,\delta u_1,\delta u_2,\delta u_3)$ in $\mathcal L_p$). 

\begin{prop}
	\label{prop:Ellip}
Let $\mathcal L$ be the operator $\mathcal L_{p\mu\rho}$ in \eqref{eq:LPMuRho}. 

\begin{enumerate}[label={\alph*)}] 
\item The principal symbol of $\mathcal L_{p\mu\rho}$ (in the dynamic case) is
\begin{multline}
\label{eq:princSymb}
	\mathcal L_0((\b x,t),\i\gb \xi)=\\
	\begin{pmatrix} 
		\i\xi_1  & 2 \i\gb \xi_s\cdot\varepsilon(\b u)_1 & -(u_1)_{tt} & -\mu\xi_1^2 -\mu|\gb \xi|^2+\rho\xi_4^2 & -\mu \xi_1\xi_2 & -\mu \xi_1\xi_3  \\
		\i\xi_2  & 2 \i\gb \xi_s\cdot\varepsilon(\b u)_2 & -(u_2)_{tt} & -\mu \xi_1\xi_2 & -\mu\xi_2^2 -\mu|\gb \xi|^2 + \rho\xi_4^2 & -\mu \xi_2\xi_3 \\
		\i\xi_3  & 2 \i\gb \xi_s\cdot\varepsilon(\b u)_3 & -(u_3)_{tt} & -\mu \xi_1\xi_3 & -\mu \xi_2\xi_3 & -\mu\xi_3^2 -\mu|\gb \xi|^2+\rho\xi_4^2 \\  
		0&0&0&1&0&0\\
		0&0&0&0&1&0\\
		0&0&0&0&0&1
	\end{pmatrix} ,
\end{multline}
where $\gb \xi_s:=(\xi_1,\xi_2,\xi_3)$ for $\gb \xi\in\mathbb{R}^4$.
		
\item 
	For every point $(\b x,t)$, the operator $\mathcal L_{p\mu\rho}$ is not overdetermined elliptic.

\item In the quasi-static case \eqref{eq:QuasiStatic}, the principal symbol of $\mathcal L_{p\mu\rho}=\mathcal L_{p\mu}$ is
\begin{multline}
\label{eq:princSymbSpat}
	\mathcal L_0(\b x,\i\gb \xi)=\\
	\begin{pmatrix} 
		\i\xi_1  & 2 \i\gb \xi\cdot\varepsilon(\b u)_1 &  -\mu\xi_1^2 -\mu|\gb \xi|^2 & -\mu \xi_1\xi_2 & -\mu \xi_1\xi_3  \\
		\i\xi_2  & 2 \i\gb \xi\cdot\varepsilon(\b u)_2 &  -\mu \xi_1\xi_2 & -\mu\xi_2^2 -\mu|\gb \xi|^2  & -\mu \xi_2\xi_3 \\
		\i\xi_3  & 2 \i\gb \xi\cdot\varepsilon(\b u)_3 &  -\mu \xi_1\xi_3 & -\mu \xi_2\xi_3 & -\mu\xi_3^2 -\mu|\gb \xi|^2 \\  
		0&0&1&0&0\\
		0&0&0&1&0\\
		0&0&0&0&1
	\end{pmatrix} 
\end{multline}
for $\gb \xi\in\mathbb{R}^3$.

\item 
		For every point $\b x$ consider $\mathcal L=\mathcal L_{p\mu\rho}=\mathcal L_{p\mu}$ in the quasi-static case \eqref{eq:QuasiStatic}. Then $\mathcal L$ is not overdetermined elliptic at $\b x$.

	\end{enumerate}
\end{prop}

\begin{proof}
a) To compute the principal symbol \eqref{eq:princSymb}, we refer to the definition of $\mathcal L_{p\mu\rho}$ in \eqref{eq:LPMuRho}.

We write the first three equations in $\mathcal L_{p\mu\rho}(\delta p,\delta\mu,\delta\rho,\delta\b u)=\mathcal S$ as
\begin{equation}
\label{eq:ReprGl}
	\partial_i \delta p+2\;\nabla\cdot(\delta\mu\;\varepsilon(\b u)_i)+2\nabla\cdot(\mu\;\varepsilon(\delta\b u)_i)-\delta\rho(u_i)_{tt}-\rho(\delta u_i)_{tt}=0,\qquad i=1,2,3
\end{equation}
where we denote the columns of the (symmetric) strain as
\[
	\varepsilon (\b u)=\frac{1}{2}(\nabla \b u +\nabla \b u^\top)=(\varepsilon(\b u)_1,\varepsilon(\b u)_2,\varepsilon(\b u)_3).
\]
The meaning of $\varepsilon(\delta\b u)_i$ is analogous.

The last three equations in $\mathcal L_{p\mu\rho}(\delta p,\delta\mu,\delta\rho,\delta\b u)=\mathcal S$ are written as
\begin{equation}
\label{eq:ReprGlInt}
	\delta u_i = \delta K_i\qquad i=1,2,3
\end{equation}

We now want to determine the entries $L_{0,ij}$ of the principal symbol $\mathcal L_0((\b x,t),\i\gb\xi)$. Note that each of the columns in $\mathcal L_0$ exactly corresponds to one of the variables $(\delta p,\delta\mu,\delta \rho,\delta u_1,\delta u_2,\delta u_3)$. Starting with $j=0$, we go through the variable list until $j=6$. For each variable, we determine the term where the unknown appears in the equations \eqref{eq:ReprGl} resp. in \eqref{eq:ReprGlInt}. Then we choose that component of the term which has order $t_j+s_i$. Substituting $\i\gb\xi$ for $\frac{\partial}{\partial(\b x,t)}$ in that component gives the corresponding entries $(L_{0,ij})_{i=1}^6$ in the $j$-th column of $\mathcal L_0$. 

For the first column corresponding to $\delta p$, the term $\partial_i \delta p$ in \eqref{eq:ReprGl} is translated to $(\i\xi_1,\i\xi_2,\i\xi_3)$ in $(L_{0,i1})_{i=1}^3$. The second column corresponding to $\delta \mu$, and the summand of highest order in the term $2\;\nabla\cdot(\delta\mu \;\varepsilon(\b u)_i)$ translates to $2\;\gb\xi_s\cdot \varepsilon(\b u)_i$ in $(L_{0,i2})_{i=1}^3$ with $\gb\xi_s=(\xi_1,\xi_2,\xi_3)$. In the third column corresponding to $\delta\rho$, no differentiation occurs, so we just have $-(u_i)_{tt}$ as entries in $(L_{0,i3})_{i=1}^3$.

The last three columns correspond to the variables $(\delta u_1,\delta u_2,\delta u_3)$. The relevant terms in \eqref{eq:ReprGl} are 
\begin{equation}
\label{eq:SubTerm}
	2\;\nabla\cdot(\mu\;\varepsilon(\delta\b u)_i)-\rho(\delta u_i)_{tt}=
	\nabla\cdot(\mu
	\begin{pmatrix}
		\partial_1 \delta u_i\\
		\partial_2 \delta u_i\\
		\partial_3 \delta u_i
	 \end{pmatrix})
	 + \nabla\cdot(\mu
	 \begin{pmatrix}
		 \partial_i \delta u_1\\
		 \partial_i \delta u_2 \\
		 \partial_i \delta u_3
	 \end{pmatrix})-\rho(\delta u_i)_{tt}.
\end{equation}
We substitute $i\gb\xi$ for differentiation in \eqref{eq:SubTerm} and take the terms of highest order to find the entries of the columns $(L_0)_{i,j},j=4,5,6$; these are the terms $-\mu\xi_i^2-\mu|\gb \xi|^2+\rho\xi_4^2$ in $(L_0)_{i,i+3}, i=1,2,3$, and the term $-\mu\xi_i\xi_j$ in the entries $(L_0)_{i,j},j=4,5,6,j\neq i$. -- The last three equations \eqref{eq:ReprGlInt} contain no derivatives and give rise to the identity matrix in the entries $(L_0)_{ij},4\leq i,j\leq 6$ of the principal symbol.

\medskip
c) The calculations for the symbol \eqref{eq:princSymbSpat} in the spatial case are exactly the same as for the spatio-temporal case. The only change in this case is that there is no temporal derivative. Therefore there is no variable $\delta\rho$, and one column less than in \eqref{eq:princSymb}, and $\xi_4$ can be set to zero everywhere.

\medskip
b) and d)
We first observe that both in \eqref{eq:princSymb} as well in \eqref{eq:princSymbSpat}, three of the columns are clearly linearly independent. The first two columns, though, can be linearly dependent:

Consider the symmetric strain $\varepsilon=\varepsilon^\top$. As the entries are real, there always exists an eigenvector $\b v$ such that 
\[
\varepsilon*\b v = \begin{pmatrix}\varepsilon_1\cdot\b v,\varepsilon_2\cdot\b v , \varepsilon_3\cdot\b v\end{pmatrix}=\kappa\b v,
\]
where $*$ denotes matrix multiplication. Choosing $(\xi_1,\xi_2,\xi_3)=\b v$ gives linear dependence of $\mathcal L_0(\i\gb\xi)$ in the first two columns.

In conclusion, at each point $(\b x,t)$ resp. $\b x$, there are  choices of $\gb \xi$ such that the symbols in \eqref{eq:princSymb} resp. \eqref{eq:princSymbSpat} do not have full rank. Therefore $\mathcal L_{p\mu\rho}$ is not overdetermined elliptic.
\end{proof}

We now restrict the focus on linearizations in only one direction, which were introduced in Section~\ref{sec:Setting}. Then the corresponding principal symbol contains fewer columns and results on ellipticity can be obtained.
\begin{cor}
\label{cor:AL}
	The operator $\mathcal L_p$ in \eqref{eq:LP}, considered in the stationary case, is elliptic everywhere.
\end{cor}
\begin{proof}
Let $\mathcal L=\mathcal L_p$.	The principal symbol consists of the first and the three last columns of the matrix in \eqref{eq:princSymbSpat}:
\begin{equation}
\label{eq:princLP}
\mathcal L_0(\b x,\i\gb \xi)=
\begin{pmatrix}
\i\xi_1  &\phantom{78}&   \\
\i\xi_2   &\phantom{78}*&  \\
\i\xi_3    &\phantom{78}& \\  
 &\b{Id_{3x3}}
\end{pmatrix}.
\end{equation}	
	For $\gb\xi\neq 0$, this symbol has maximal rank 4 everywhere, therefore $\mathcal L_p$ is elliptic.
\end{proof}

\begin{cor}
\label{cor:AM}
	The operator $\mathcal L_\mu$ in \eqref{eq:LMu}, considered in the stationary case, is elliptic exactly at points which satisfy
	\begin{equation}
	\label{eq:NonSingStrain}
		\det(\varepsilon(\b u(\b x,t))\neq 0.
	\end{equation}
	The operator $\mathcal L_\mu^s$, corresponding to $s$ measurements $(\delta\b u)_k, 1\leq k\leq s$, is elliptic exactly at points where at least one of $\det(\varepsilon(\b u_k(\b x,t))\neq 0$.
\end{cor}
\begin{proof}
Let $\mathcal L = \mathcal L_\mu$ (the case of one measurement). Then the principal symbol consists of the second and the three last columns of the matrix in \eqref{eq:princSymbSpat}:
\begin{equation}
\label{eq:princLMu}
\mathcal L_0(\b x,\i\gb \xi)=
\begin{pmatrix}
\i\gb \xi_s\cdot\varepsilon(\b u)_1  &\phantom{78}&   \\
\i\gb \xi_s\cdot\varepsilon(\b u)_2   &\phantom{78}*&  \\
\i\gb \xi_s\cdot\varepsilon(\b u)_3   &\phantom{78}& \\  
 &\b{Id_{3x3}}
\end{pmatrix}
\end{equation}		
	This symbol has rank 4 provided that the first column is non-degenerate, which is equivalent to the condition  \eqref{eq:NonSingStrain} of non-singular strain.
	
Now let $\mathcal L=\mathcal L_\mu^{(s)}$ (the case of multiple measurements). This can be treated by induction on $s$. Let the principal symbol $\mathcal L_0^{(s-1)}$ corresponding to $s-1$ measurements have dimension $a\times b$. Adding one measurement means adding three more columns (and six new lines) in the matrix, such that it has dimension $(a+6)\times (b+3)$. These three new columns are independent because of the identity component in $(L_{0,ij}), a+4\leq i\leq a+6,b+1\leq j\leq b+3$.

In the first column of $\mathcal L_0^s$, there are three new entries $2\i\gb\xi_s\cdot\varepsilon(\b u_s)_{i=1}^3$ at position $(L_{0,i1})_{i=3s-2}^{3s}$. If we have that for $1\leq k\leq s$, at least one of $\det(\varepsilon(\b u_k(\b x,t))\neq 0$ is non-zero, then the principal symbol $\mathcal L_0^s$ has full rank: in the case $k<s$ because of the induction assumption, and in the case $k=s$ because of non-degeneracy in the first column due to the new measurement.
\end{proof}

\begin{cor}
\label{cor:ARho}
	The operator $\mathcal L_\rho$ in \eqref{eq:LRho} is elliptic exactly at points with $\b u_{tt}\neq 0$. For the case of several measurements, $\mathcal L_\rho^{(s)}$ is elliptic exactly at points $(\b x,t)$ where at least one of $(\b u_k)_{tt}(\b x,t)\neq 0$, $1\leq k\leq s$.
\end{cor}
\begin{proof}
Let $\mathcal L=\mathcal L_\rho$.	The principal symbol consists of the last four columns of the matrix in \eqref{eq:princSymb}:
\begin{equation}
\label{eq:princLRho}
\mathcal L_0((\b x,t),\i\gb \xi)=
\begin{pmatrix}
-(u_1)_{tt}  &\phantom{78}&   \\
-(u_2)_{tt}   &\phantom{78}*&  \\
-(u_3)_{tt}   &\phantom{78}& \\  
 &\b{Id_{3x3}}
\end{pmatrix}
\end{equation}	
	 This symbol has rank 4 iff $\b u_{tt}$ is nonzero. The statement for multiple measurements is proved by induction, analogous to the case of $\mathcal L_\mu^s$.
\end{proof}

\begin{cor}
\label{cor:ApMu}
	The operator $\mathcal L_{p\mu}$ is not elliptic at points $(\b x,t)\in\Omega\times T$ in the spatio-temporal case, resp. points $\b x\in\Omega$ in the spatial case.
\end{cor}
\begin{proof}
Let $\mathcal L=\mathcal L_{p\mu}$ in the dynamic case.	The principal symbol consists of the two first and the three last columns of the matrix in \eqref{eq:princSymb}:
\begin{equation}
\label{eq:princLPMu}
\mathcal L_0((\b x,t),\i\gb \xi)=
\begin{pmatrix}
\i\xi_1 &\i\gb \xi_s\cdot\varepsilon(\b u)_1  &\phantom{78}&   \\
\i\xi_2 &\i\gb \xi_s\cdot\varepsilon(\b u)_2   &\phantom{78}*&  \\
\i\xi_3 &\i\gb \xi_s\cdot\varepsilon(\b u)_3   &\phantom{78}& \\  
 &&\b{Id_{3x3}}
\end{pmatrix}
\end{equation}	
	
Fix a point $(\b x,t)$. As in the proof of Proposition~\ref{prop:APAMu}c), choose an eigenvector $\b v$ of $\varepsilon$ and let $\gb\xi$ such that $(\xi_1,\xi_2,\xi_3)=\b v$. With this choice of $\gb\xi$, the principal symbol \eqref{eq:princLPMu} has two columns which are linearly dependent. Therefore, the operator $\mathcal L_{p\mu}$ is not overdetermined elliptic.
\end{proof}

\begin{remark}
\label{rem:PL}
As starting point of our analysis, we have choosen the modified model \eqref{eq:ModModel} with the substitution $p=\lambda\nabla\cdot\b u$ in \eqref{eq:Pressure}. Of course, we could also analyze the system corresponding to the original model~\eqref{eq:Modell}. Using linearization in direction of $\delta\lambda$, we have the operator $\mathcal L_{\lambda}$ in \eqref{eq:LLambda}.

Now let $\mathcal L=\mathcal L_\lambda$ in the quasi-static case. Doing calculations as in Proposition~\ref{prop:Ellip}, and restricting the focus on the first and last three columns as in Corollary~\ref{cor:AL}, one finds the resulting principal symbol as 

\begin{equation}
\label{eq:SymbL}
	\mathcal L_0(\b x,\i\gb \xi)=
	\begin{pmatrix}
	\i\xi_1\nabla\cdot\b u   &\phantom{78}&   \\
	\i\xi_2\nabla\cdot\b u    &\phantom{78}*&  \\
	\i\xi_3\nabla\cdot\b u   &\phantom{78}& \\  
 		&\b{Id_{3x3}}
	\end{pmatrix}.
\end{equation}
In this case, ellipticity is harder to obtain, since it is tied up with non-vanishing of $\nabla\cdot\b u$. This actually may be a condition which is harder to guarantee in experiments, see the discussion in Section~\ref{sec:Disc}.

\end{remark}
\subsection{Lopatinskii condition}
\label{sec:Lopat}
We want to use the stability criterion in Theorem \ref{thm:Sol} for the systems $\mathcal A_p$, $\mathcal A_\mu$ and $\mathcal A_\rho$ in \eqref{eq:SystemP}, \eqref{eq:SystemMu}, \eqref{eq:SystemRho}. For this purpose, we are checking the covering condition in Definition \ref{def:Lopat} for the various differential operators $\mathcal L$ introduced in Section~\ref{sec:Setting}, together with the relevant boundary data where necessary.

\begin{prop}
\label{prop:APAMu}
	The systems $\mathcal L_p$ and $\mathcal L_\mu$ in \eqref{eq:LP} and \eqref{eq:LMu} satisfy the Lopatinskii condition with arbitrary boundary data.
\end{prop}
\begin{proof}
For checking the Lopatinskii (or covering) condition in Definition \ref{def:Lopat}, we have to consider the vector space of functions satisfying the system \eqref{eq:Lopat} and show that it is trivial.

Let $\mathcal L=\mathcal L_p$ with principal symbol \eqref{eq:princLP}, and let $\b y$ be a point on the boundary. Then the system of equations $\mathcal L_0(\b y,\i\gb \zeta+\gb\nu\frac{d}{dz})\tilde{\b u}(z)=0$ in \eqref{eq:Lopat} comprises 6 equations for 4 unknowns $(\tilde u_1(z),\tilde u_2(z),\tilde u_3(z),\tilde u_4(z))$.

In the entries $(L_0)_{ij}, 4\leq i\leq 6, 2\leq j\leq 4$ of \eqref{eq:princLP}, there is a 3-by-3 identity matrix. The three last equations of \eqref{eq:Lopat} therefore mean that 
\begin{equation}
\label{eq:u123Vanish}
	\tilde u_2(z)=\tilde u_3(z)=\tilde u_4(z) = 0.
\end{equation}
The first three equations in \eqref{eq:Lopat} then reduce to
\begin{equation}
\label{eq:LopatMu}
	\begin{aligned}
		(\i\zeta_1+\nu_1\frac{d}{dz})\tilde u_1(z) &=0 \\
		(\i\zeta_2+\nu_2\frac{d}{dz})\tilde u_1(z) &=0 \\
		(\i\zeta_3+\nu_3\frac{d}{dz})\tilde u_1(z) &=0.
	\end{aligned}
\end{equation}
Any functions which are in the vector space considered in Definition~\ref{def:Lopat} have to satisfy \eqref{eq:LopatMu}. 

The only possible solutions of \eqref{eq:LopatMu} consist of functions of form $e^{\i\lambda z}$, where the parameter $\lambda=\frac{\zeta_i}{\nu_i}$ is a real number; neither of these functions tends to $0$ for $z\to\infty$. In this argument, we did not use any boundary constraint. Therefore the vector space to be considered in Definition~\ref{def:Lopat} is trivial for every $\b y\in\partial\Omega$, and the Lopatinskii covering condition is always satisfied.

Now let $\mathcal L=\mathcal L_\mu$ with principal symbol \eqref{eq:princLMu}. Then, for $\b y$ on the boundary, consider the system $\mathcal L_0(\b y,\i\gb \zeta+\gb\nu\frac{d}{dz})\tilde{\b u}(z)=0$ in \eqref{eq:Lopat}. Similarly to \eqref{eq:u123Vanish}, the last three components vanish, and the system reduces to
\begin{equation}
\label{eq:SubGlgLopatMu}
\begin{aligned}
	\b g_1\cdot(\i\gb \zeta_s+\gb\nu\frac{d}{dz})\tilde u_1(z)&=0 \\
	\b g_2\cdot(\i\gb \zeta_s+\gb\nu\frac{d}{dz})\tilde u_1(z)&=0 \\
	\b g_3\cdot(\i\gb \zeta_s+\gb\nu\frac{d}{dz})\tilde u_1(z)&=0.
\end{aligned}
\end{equation}
Here, we use fixed vectors $\b g_j=\varepsilon(\b u(\b y))_j, j=1,2,3$ and $\gb\zeta_s=(\zeta_1,\zeta_2,\zeta_3)$.

Suppose that the relation $\b g_j\cdot \gb\nu\neq 0$ holds for all $j$. Then we have solutions of \eqref{eq:SubGlgLopatMu} of form $e^{\i\frac{\b g_i\cdot\gb\zeta_s}{\b g_i\cdot\gb \nu}z}$. The numbers $\frac{\b g_i\cdot\gb\zeta_s}{\b g_i\cdot\gb \nu}$ are real, so neither of these functions  tends to $0$. 

Suppose, on the other hand, that we have $\b g_{j_0}\cdot\gb\nu=0$ for one $j_0$. Then, using $\gb\nu\cdot\gb\zeta_s=0$, we have $\b g_{j_0}\cdot\gb\zeta_s\neq 0$. Inserting this information in \eqref{eq:SubGlgLopatMu}, we directly get $\tilde u_1(z)=0$.

Therefore the Lopatinskii condition for $\mathcal L_\mu$ is satisfied with arbitrary boundary data.
\end{proof}

\begin{prop}
\label{prop:ARho}
	The system $\mathcal L_\rho$ in \eqref{eq:LRho} satisfies the Lopatinskii condition with arbitrary boundary data if and only if $\b u(\b y,t)_{tt}\neq 0$ for all $(\b y,t)\in\partial(\Omega\times T)$.
\end{prop}
\begin{proof}
Consider the system $\mathcal L_0(\b y,\i\gb \zeta+\gb\nu\frac{d}{dz})\tilde{\b u}(z)=0$ in \eqref{eq:Lopat} for the operator $\mathcal L_\rho$. As in the proof of Proposition \ref{prop:APAMu}, we get 
\[
	\tilde u_2(z)=\tilde u_3(z)=\tilde u_4(z)=0.
\]
The equations for $\tilde u_1(z)$ therefore reduce to
\begin{equation}
\label{eq:LopatRho}
	(u_i)_{tt} \tilde u_1(z) = 0 \qquad \text{for } i = 1,2,3
\end{equation}
If we suppose $\b u(\b y,t)_{tt}\neq 0$, then \eqref{eq:LopatRho} implies $\tilde u_1(z)=0$, therefore the Lopatinskii condition is satisfied.

Conversely, suppose that $\b u(\b y,t)_{tt}=0$ for one $(\b y,t)$. Then one can choose $\b{\tilde u}(z)=(u_1(z),0,0,0)$ satisfying \eqref{eq:LopatRho} for any function $u_1$ that satisfies $u_1(z)\to 0$. Therefore the Lopatinskii condition would be violated.
\end{proof}

\begin{prop}
\label{prop:APMu}
	The system $\mathcal L_{p\mu}$ in \eqref{eq:LPMu} satisfies the Lopatinskii condition with boundary data \eqref{eq:Boundary}, \eqref{eq:BoundNew} at $\b y\in\partial\Omega$ provided that the unit normal vector $\gb \nu(\b y)$ is not an eigenvector of $\varepsilon(\b u(\b y))$.
\end{prop}
\begin{proof}
Let $\mathcal L=\mathcal L_{p\mu}$ with principal symbol \eqref{eq:princLPMu}.
Consider the system of equations 
\[ 
	\mathcal L_0(\b y,\i\gb \zeta+\gb\nu\frac{d}{dz})\tilde{\b u}(z)=0
\]
from \eqref{eq:Lopat} for the vector $\tilde{\b u}(z)=(\tilde u_1(z),\ldots,\tilde u_5(z))$. The last three equations of this system yield that $\tilde u_3(z)=\tilde u_4(z)=\tilde u_5(z)=0$, similar to \eqref{eq:u123Vanish}. For the two remaining functions $\tilde u_1(z)$ and $\tilde u_2(z)$, the equations reduce to
\begin{align}
\label{eq:Lop1}
	(\i\zeta_1+\nu_1\frac{d}{dz})\tilde u_1(z) + \b g_1\cdot(\i\gb \zeta_s+\gb\nu\frac{d}{dz})\tilde u_2(z)&=0 \\
\label{eq:Lop2}
	(\i\zeta_2+\nu_2\frac{d}{dz})\tilde u_1(z) + \b g_2\cdot(\i\gb \zeta_s+\gb\nu\frac{d}{dz})\tilde u_2(z)&=0 \\
\label{eq:Lop3}
	(\i\zeta_3+\nu_3\frac{d}{dz})\tilde u_1(z) + \b g_3\cdot(\i\gb \zeta_s+\gb\nu\frac{d}{dz})\tilde u_2(z)&=0 ,
\end{align}
Here, $\gb\zeta_s=(\zeta_1,\zeta_2,\zeta_3)$ and we use the fixed vectors
\begin{equation}
\label{eq:DefG}
	\b g_j=\varepsilon(\b u(\b y))_j, j=1,2,3.
\end{equation}
Now let us assume that there exist non-zero solutions to this system. By elimination of $\tilde u_2(z)$ and $\frac{d}{dz}\tilde u_2(z)$ from \eqref{eq:Lop1}-\eqref{eq:Lop2}, \eqref{eq:Lop1}-\eqref{eq:Lop3} and \eqref{eq:Lop2}-\eqref{eq:Lop3}, respectively, we find the following three equations which $\tilde u_1(z)$ has to satisfy:
\begin{align}
\label{eq:LopEl1}
	a_{1,2} \frac{d^2}{dz^2}\tilde u_1(z)+b_{1,2} \frac{d}{dz}\tilde u_1(z)+c_{1,2}\tilde u_1(z)&=0 \\
\label{eq:LopEl2}
	a_{1,3} \frac{d^2}{dz^2}\tilde u_1(z)+b_{1,3} \frac{d}{dz}\tilde u_1(z)+c_{1,3}\tilde u_1(z)&=0 \\
\label{eq:LopEl3}
	a_{2,3} \frac{d^2}{dz^2}\tilde u_1(z)+b_{2,3} \frac{d}{dz}\tilde u_1(z)+c_{2,3}\tilde u_1(z)&=0,
\end{align}
with the coefficients
\begin{align}
\label{eq:KoeffLop1}
	a_{p,q} &= \nu_p\b g_q\cdot\gb\nu-\nu_q\b g_p\cdot\gb\nu \\
\label{eq:KoeffLop2}
	b_{p,q} &=\i(\zeta_p \b g_q\cdot \gb\nu + \nu_p\b g_q\cdot\gb\zeta_s-\zeta_q\b g_p\cdot\gb\nu-\nu_q\b g_p\cdot\gb\zeta_s) \\
\label{eq:KoeffLop3}
	c_{p,q} &= \zeta_q\b g_p\cdot\gb\zeta_s-\zeta_p\b g_q\cdot\gb\zeta_s.
\end{align}
The same equations are obtained for $\tilde u_2(z)$ also, by elimination of $\tilde u_1(z)$ from \eqref{eq:Lop1}-\eqref{eq:Lop3}.

Let us consider the matrix of the coefficients in \eqref{eq:LopEl1}-\eqref{eq:LopEl3},
\begin{equation}
\label{eq:MatrixA}
	A=
	\begin{pmatrix}
		a_{1,2} & b_{1,2} & c_{1,2} \\
		a_{1,3} & b_{1,3} & c_{1,3} \\
		a_{2,3} & b_{2,3} & c_{2,3} \\
	\end{pmatrix}.
\end{equation}

We claim that under our assumption, the system \eqref{eq:LopEl1}-\eqref{eq:LopEl3} is nontrivial, which is equivalent to $A\neq 0$.

Assume, on the contrary, that all entries in $A$ vanish: 
\begin{equation}
\label{eq:AVanish}
A=0.
\end{equation}
Incorporating the information \eqref{eq:KoeffLop1}-\eqref{eq:KoeffLop3}, the nine equations in \eqref{eq:AVanish} can be written in matrix form as 
\begin{equation}
\label{eq:NonDegSys}
	\begin{pmatrix}
		0&0&0&-\nu_2&\nu_1&0\\
		0&0&0&-\nu_3&0&\nu_1\\
		0&0&0&0&-\nu_3&\nu_2\\
		-\nu_2&\nu_1&0&-\zeta_2&\zeta_1&0\\
		-\nu_3&0&\nu_1&-\zeta_3&0&\zeta_1\\
		0&-\nu_3&\nu_2&0&-\zeta_3&\zeta_2\\
		\zeta_2&-\zeta_1&0&0&0&0\\
		\zeta_3&0&-\zeta_1&0&0&0\\
		0&\zeta_3&-\zeta_2&0&0&0
	\end{pmatrix}
	*
	\begin{pmatrix}
		\b g_1\cdot\gb\zeta_s\\
		\b g_2\cdot\gb\zeta_s\\
		\b g_3\cdot\gb\zeta_s\\
		\b g_1\cdot\gb\nu\\
		\b g_2\cdot\gb\nu\\
		\b g_3\cdot\gb\nu
	\end{pmatrix}
	= 0
\end{equation}
Here, $*$ denotes matrix multiplication.

The system \eqref{eq:NonDegSys} can be seen as linear system of equations for the unknown variables $\{\b g_1\cdot\gb\zeta_s,\;\b g_2\cdot\gb\zeta_s,\;\b g_3\cdot\gb\zeta_s,\;\b g_1\cdot\gb\nu,\;\b g_2\cdot\gb\nu,\;\b g_3\cdot\gb\nu\}$. The system can then be solved by elimination. The one dimensional solution space is generated by the vector $(\zeta_1,\zeta_2,\zeta_3,\nu_1,\nu_2,\nu_3)$. In particular, we have
\[
	\b g_j\cdot\gb \nu=\kappa \nu_j\qquad\text{ for }j=1,2,3.
\]
Using the definition of $\b g_j$ in \eqref{eq:DefG}, it follows that 
\begin{equation}
\label{eq:EigenvectEps}
	\varepsilon(\b u(\b y))*\gb \nu(\b y)=\kappa\gb \nu(\b y),
\end{equation}
as we have that $\varepsilon=\varepsilon^\top$. But, by assumption, $\gb \nu (\b y)$ must not be an eigenvector of the strain $\varepsilon(\b u(\b y))$, so this cannot happen. Therefore \eqref{eq:AVanish} is wrong and $A\neq 0$.

Consequently, the equations \eqref{eq:LopEl1}-\eqref{eq:LopEl3} are always nontrivial in our case. As these are linear ordinary differential equations of second order, the basis of solutions consists of functions of form $e^{\lambda z}$ and $z\;e^{\lambda z}$.

We now claim that there can be no solutions of form $z\;e^{\lambda z}$ to the differential equations \eqref{eq:Lop1}-\eqref{eq:Lop3}.

Assume, on the contrary, that a solution of form $z\;e^{\lambda z}$ exists. According to the theory of linear ODE, $\lambda$ is a double zero of the three characteristic polyomials 
\[
	a_k\lambda^2+b_k\lambda + c_k = 0 \qquad \text{for } k=(1,2),(1,3),(2,3).
\]
The discriminant has to vanish, so we have $\lambda = \frac{-b_k\pm\sqrt{b_k^2-4 a_k c_k}}{2a_k}=\frac{-b_k}{2a_k}$, thus $b_k=-2\lambda a_k$ for $k=(1,2),(1,3),(2,3)$; consequently we also have $c_k = 4\lambda^2 a_k$. Because of these proportionalities between $a_k$, $b_k$ and $c_k$, the matrix $A$ in \eqref{eq:MatrixA} has rank~1. Therefore we have also a proportionality between the rows in $A$, i.e. the relations
\begin{equation}
\label{eq:LopSystem}
\begin{aligned}
a_{1,3}&=\gamma \;a_{1,2}\\
a_{2,3}&=\delta \;a_{1,2}
\end{aligned}
\qquad
\begin{aligned}
b_{1,3}&=\gamma \;b_{1,2}\\
b_{2,3}&=\delta \;b_{1,2}
\end{aligned}
\qquad
\begin{aligned}
c_{1,3}&=\gamma \;c_{1,2}\\
c_{2,3}&=\delta \;c_{1,2}
\end{aligned}
\end{equation}
for some constants $\gamma,\delta$. Using \eqref{eq:KoeffLop1}-\eqref{eq:KoeffLop3}, the equations in \eqref{eq:LopSystem} can be written in matrix form as
\begin{equation}
\label{eq:NoDoubleValueSystem}
	\begin{pmatrix}
		0&0&0&-\nu_3+\gamma \nu_2 & -\gamma \nu_1 & \nu_1 \\
		0&0&0&0& \delta \nu_2 & -\nu_3-\delta \nu_1 & \nu_2 \\
		-\nu_3+\gamma\nu_2&-\gamma\nu_1&\nu_1&-\zeta_3+\gamma\zeta_2&-\gamma\zeta_1&\zeta_1\\
		\delta\nu_2&-\nu_3-\delta\nu_1&v_2&+\delta\zeta_2&-\zeta_3-\delta\zeta_1&\zeta_2\\
		\zeta_3-\gamma\zeta_2&+\gamma\zeta_1&-\zeta_1&0&0&0\\
		-\delta\zeta_2&\zeta_3+\delta\zeta_1&-\zeta_2&0&0&0
	\end{pmatrix}
	*
	\begin{pmatrix}
		\b g_1\cdot\gb\zeta_s\\
		\b g_2\cdot\gb\zeta_s\\
		\b g_3\cdot\gb\zeta_s\\
		\b g_1\cdot\gb\nu\\
		\b g_2\cdot\gb\nu\\
		\b g_3\cdot\gb\nu
	\end{pmatrix}
	= 0,
\end{equation}
which we interpret, as in \eqref{eq:NonDegSys}, as linear system for the unknown variables $\{\b g_1\cdot\gb\zeta_s,\;\b g_2\cdot\gb\zeta_s,\;\b g_3\cdot\gb\zeta_s,\;\b g_1\cdot\gb\nu,\;\b g_2\cdot\gb\nu,\;\b g_3\cdot\gb\nu\}$. Now the matrix in \eqref{eq:NoDoubleValueSystem} has rank~6, so the equations reduce to 
\begin{equation}
\label{eq:ReducSysLopat}
\begin{aligned}
\b g_j\cdot\gb\zeta_s&=0\qquad j=1,2,3\qquad\text{ and}\\
\b g_j\cdot\gb\nu &= 0\qquad j=1,2,3.
\end{aligned}
\end{equation}
As in \eqref{eq:EigenvectEps}, it follows from the second equation in \eqref{eq:ReducSysLopat} that $\gb\nu$ is an eigenvector of the strain $\varepsilon(\b u(\b y))$ corresponding to the eigenvalue $0$.  But this has been ruled out by hypothesis.

So the only solutions of \eqref{eq:Lop1}-\eqref{eq:Lop3} are of form $e^{\lambda z}$. The boundary condition \eqref{eq:BoundNew} then leads to the unique solution of $\tilde u_1(z)=0$. The same chain of arguments can be invoked to obtain $\tilde u_2(z)=0$. Therefore the Lopatinskii condition is satisfied at $\b y$.
\end{proof}

\subsection{Stability}
In Sections \ref{sec:Ell} and \ref{sec:Lopat}, we collected results about the systems $\mathcal A_p$, $\mathcal A_\mu$ and $\mathcal A_\rho$ defined in \eqref{eq:SystemP}, \eqref{eq:SystemMu}, \eqref{eq:SystemRho}; this now allows to derive the following information about the stability properties of these operators, and about reconstruction of parameters in the linearized inverse problem of quantitative elastography.

We choose numbers $p=2, l$ with $l>0, pl>n$, and a bounded and connected domain $\Omega$ with $C^{l+2}$-boundary. 

In the following statements, we suppose that we have a reference state $(p,\mu,\rho)$ with $\mu\in H^{l+1}(\Omega)$, $\rho\in H^{l}(\Omega)$ and for the reference displacement $\b u = \mathcal V(p,\mu,\rho)$ in \eqref{eq:RefDispl} we require $\b u\in H^{l+2}(\Omega)$, $\b u_{tt}\in H^{l}(\Omega)$. The existence of such a displacement field $\b u$ can be ensured by \cite[Thm.8.1]{MclThoYoo11} if we require that the material parameter are in a H\"older space.

Note that in our analysis, we consider $(\delta p,\delta\mu,\delta\rho,\delta u_1,\delta u_2,\delta u_3)$ the parameters and the solutions as variables in the system.  For subsequent analysis, we use the Douglis-Nirenberg numbers $(t_j)_{j=1}^6=(1,1,0,2,2,2)$, corresponding to the variables $(\delta p,\delta\mu,\delta\rho,\delta u_1,\delta u_2,\delta u_3)$, as well as $(s_i)_{i=1}^6=(0,0,0,-2,-2,-2)$. -- Note that for the particular operators $\mathcal A_\mu, \mathcal A_\rho, \mathcal A_p,\mathcal A_\lambda$, only four of these variables are used in the corresponding system.

We consider the solutions of the linearized versions as follows: $\delta p\in H^{l+1}(\Omega)$, $\delta \mu\in H^{l+1}(\Omega)$, $\delta \rho\in H^{l}(\Omega)$, $\delta u_i\in H^{l+2}(\Omega)$ for $i=1,2,3$. The inhomogeneity $\delta\b K$ is supposed to lie in $H^{l+2}(\Omega)^3$. The possibility of choosing these functions also follows from \cite[Thm.8.1]{MclThoYoo11}.

In the following theorem, we have recourse to the notion of an operator having a \emph{left regularizer}, as in Theorem \ref{thm:Sol}, see also \cite{Kre82}. Other names for this notion are that $A$ is a \emph{left semi-Fredholm operator} \cite[Ch. XI, \S 2]{Con90}.

\begin{theorem}
\label{thm:StabMu}
The operator $\mathcal A_\mu^{(r)}$ for $r\geq 1$ measurements, described in \eqref{eq:SystemMu2} and considered in the quasi-static case, has a left regularizer on  $\Omega$ precisely when, for all $\b x\in\Omega$,
\begin{equation}
\label{eq:NonSingStrainT}
	\det(\varepsilon(\b u_k(\b x)))\neq 0\qquad \text{for at least one measurement } 1\leq k\leq r.
\end{equation}
Then the stability estimate
\begin{equation}
\label{eq:EstAMu}
	 \| \delta \mu \|_{H^{l+1}(\Omega)/K_1} \leq \;C \sum_{k=1}^r \| \delta\b u_k \|_{H^{l+2}(\Omega)} 	
\end{equation}
holds with a finite-dimensional kernel $K_1$.
\end{theorem}

\begin{proof}
We first treat the case of $\mathcal A_\mu$. 

Recall that we are treating the equation
\[
	\mathcal A_\mu (\delta\mu,\delta \b u) = (\mathcal S,\varphi),
\] 
and that the operator 
\[
	\mathcal A_\mu = \mathcal L_\mu \times \mathcal B: D(p,l)\to R(p,l)
\]
as described in \eqref{eq:SystemMu}, Section~\ref{sec:Setting}, has 6 equations in the interior and four variables $(\delta \mu,\delta u_1,\delta u_2,\delta u_3)=(\delta \mu,\delta\b u)$. The inhomogeneity is $\mathcal S=(0,0,0,\delta K_1,\delta K_2, \delta K_3)^\top$ and $\varphi = 0$.

The choice of Douglis-Nirenberg numbers corresponding to these variables is $(t_j)_{j=1}^{4}=(1,2,2,2)$ and $(s_i)_{i=1}^6=(0,0,0,-2,-2,-2)$, $(\sigma_k)_{k=1}^3 = (-2,-2,-1)$, and the principal symbol of $\mathcal L_\mu$ is \eqref{eq:princLMu}. Therefore, we have (according to \eqref{eq:Dpl} and \eqref{eq:Rpl}) the spaces
\begin{align*}
	D(p,l)&=H^{l+1}(\Omega)\times (H^{l+2}(\Omega))^3\\
	R(p,l)&=H^{l}(\Omega)\times (H^{l+2}(\Omega))^3\times H^{l+\frac{5}{2}}(\Omega)\times H^{l+\frac{5}{2}}(\Omega)\times H^{l+\frac{3}{2}}(\Omega).
\end{align*}

Given the assumption \eqref{eq:NonSingStrainT} on the determinant of the strain $\varepsilon(\b u)$ of the reference state, $\mathcal L_\mu$ is elliptic according to Corollary \ref{cor:AM}, and satisfies the Lopatinskii condition according to Proposition \ref{prop:APAMu}. Therefore, the condition 1 of Theorem~\ref{thm:Sol} is satisfied for $\mathcal A_\mu$, therefore the other two equivalent conditions 2-3 are also satisfied.

Theorem~\ref{thm:Sol}.2 implies the existence of a bounded operator $R_\mu: R(p,l)\to D(p,l)$ with 
	\begin{equation}
	\label{eq:Kern}
		\mathcal R_\mu\mathcal A_\mu=\mathcal I-\mathcal T_\mu
	\end{equation}
and compact $\mathcal T_\mu:R(p,l)\to D(p,l)$.
	
The spectral theory for compact operators asserts that the kernel of $\mathcal I-\mathcal T_\mu$ is finite-dimensional \cite[Ch.VII, Thm. 7.1]{Con90}. By \eqref{eq:Kern}, the kernel $K=\ker(\mathcal A_\mu)$ is a subspace of $\ker(\mathcal I-\mathcal T_\mu)$. Consequently, $K$ is finite-dimensional also, and hence closed. Therefore we can consider the quotient $D(p,l)/K$ as a Hilbert space.

Existence of the left regularizer in \eqref{eq:Kern} implies that $\ran(\mathcal A_\mu)$ is closed \cite[Ch.XI, Thm. 2.3(ii)]{Con90}. Therefore $\ran(\mathcal A_\mu)$ is a Hilbert space. We apply the open mapping theorem \cite[Ch.III, Thm. 12.1]{Con90} to find that the inverse $A_\mu^{-1}:\ran(\mathcal A_\mu)\to  D(p,l)/K$ is continuous. Therefore, we have the existence of a real number $C$ such that for all $(\delta \mu, \delta \b u)\in D(p,l)$, the estimate
\begin{equation}
\label{eq:IsoMetr1a}
\begin{aligned}
	\| (\delta \mu, \delta \b u) \|_{D(p,l)/K} &= \| \mathcal A_\mu^{-1} A_\mu (\delta \mu, \delta \b u) \|_{D(p,l)/K} \\
	& \leq C \| A_\mu (\delta \mu, \delta \b u) \|_{R(p,l)}= C\| (\mathcal S,\varphi)\|_{R(p,l)}.
\end{aligned}
\end{equation}
holds. By \eqref{eq:LinEqns}, we have $\delta\b K=\delta\b u$ in $\mathcal S$. We set $K=K_1\times K_2$ with $K_1\subset H^{l+1}(\Omega)$. Then from \eqref{eq:IsoMetr1a}, we obtain the estimate \eqref{eq:EstAMu} for $r=1$.

The case of $\mathcal A_\mu^{(r)}$ follows straight-forwardly by induction.\\

\end{proof}

\begin{remark}Theorem~\ref{thm:StabMu} gives the criterion \eqref{eq:NonSingStrainT}, which means that at each point, at least one of the measured elastic displacement fields has \emph{non-singular strain}. The requirement of such qualitative conditions for the solutions is typical for the coupled-physics literature. In fact, the condition \eqref{eq:NonSingStrainT} for $r=2$ is a generalization of the invertibility condition for the nonlinear reconstruction problem in elastography, which was found in the research of \cite{BalBelImpMon13}, namely 
\begin{equation}
\label{eq:BalInvert}
	\det (t_2\varepsilon(\b u_1)^D - t_1\varepsilon(\b u_2)^D) \neq  0.
\end{equation}
Here, we have for $k=1,2$ that $t_k:=\tr(\varepsilon(\b u_k))=\nabla\cdot \b u_k$ and that $\b \varepsilon(\b u_k)^D := \varepsilon(\b u_k)-\frac{t_k}{3}\b {Id}$ is the deviatoric part of the strain.
It can be verified by simple calculation that violation of \eqref{eq:NonSingStrainT} leads to violation of \eqref{eq:BalInvert}.

It is unknown whether \eqref{eq:BalInvert} can be ensured with two vector fields for every distribution of material parameters. For several parameter classes, existence of boundary conditions ensuring \eqref{eq:BalInvert} can be justified, see the discussion and examples in \cite[Sec.3.3]{BalBelImpMon13}. As \eqref{eq:NonSingStrainT} is a consequence of \eqref{eq:BalInvert}, the special argumentation for \eqref{eq:BalInvert} can also be invoked for arguing for the premise of Theorem \ref{thm:StabMu} in our case.

\end{remark}

We now give an explicit characterization of the kernel of the operator $\mathcal A_\mu^{(1)}$, which will be exploited in Corollary \ref{cor:InjAMu2} to show injectivity of $\mathcal A_\mu^{(2)}$, the operator corresponding to two measurements.

\begin{theorem}
\label{thm:KerMu}
Consider $\mathcal A_\mu^{(1)}$, and suppose that the condition \eqref{eq:NonSingStrainT} with $r=1$ holds. Then the estimate \eqref{eq:EstAMu} holds with a one-dimensional kernel $K_1$. The subspace $K_1$ is generated by the element
\begin{equation}
\label{eq:KerLMu1}
	\exp\int_{\b p}^\b x \b a(\b y)d\b y
\end{equation}
with fixed $\b p\in \Omega$. Here, the vector field $\b a(\b x)$ is uniquely determined by 
\begin{equation}
\label{eq:KerLMu2}
	\b a\cdot \varepsilon(\b u)_i=-\nabla\cdot \varepsilon(\b u)_i,\qquad i=1,2,3,
\end{equation}
where $\b u$ is a reference state for which \eqref{eq:NonSingStrainT} holds.
\end{theorem}

\begin{proof}
In the proof of the statement, we derive a representation for $\delta\mu$ on a connected set and then infer that the representation is valid on $\Omega$ by a topological argument.

Suppose, to begin with,
 that $(\delta\mu,\delta\b u)\in D(p,l)$ is in the kernel of $\mathcal A_\mu=\mathcal L_\mu\times \mathcal B$.

As we have $pl>n$, the Sobolev imbedding theorems (see \cite[Thm.5.4.C]{Ada75}) imply that $\delta\mu\in H^{l+1}(\Omega)$ is continuously differentiable on $\Omega$. In particular, we have that the set
\begin{equation}
\label{eq:ThmKerA}
	A = \{ \b x\in\Omega:\delta\mu(\b x)\neq 0\}
\end{equation}
is open in $\Omega$.

If $\delta \mu \equiv 0$, then the assertion is trivially satisfied. Otherwise, there exists 
a point $\b p \in A$. In this case, consider the connected component $V$ of $\b p$ in the topology of $A\subset\Omega$, that is 
\[
	V = \bigcup\{ U: \b p\in U\subset A \text { with } U \text{ connected in } A \}.
\]
Lemma~\ref{lem:top1} implies that $V\subset A\subset\mathbb{R}^n$ is open; therefore, $V$ is also path-connected.

Suppose now $\b x\in V$. We analyze the 6 equations 
\[
	\mathcal L_\mu(\delta\mu,\delta\b u)=
	\begin{pmatrix}
		2\nabla\cdot(\delta\mu\;\varepsilon(\b u))+2\nabla\cdot(\mu\;\varepsilon(\delta\b u))\\
		\delta\b u
	\end{pmatrix}=0\qquad\text{on } V.
\]
From the last three equations, we immediately get $\delta\b u|_V= 0$. From the first three equations, we then get that
\[
	\nabla\cdot(\delta\mu\;\varepsilon(\b u)) = 0\qquad\text{on } V 
\] 
for the element $\delta\mu$, and hence
\begin{equation}
\label{eq:KerPro1}
	\nabla \delta\mu\cdot\varepsilon(\b u)_i = -\delta \mu\nabla\cdot \varepsilon(\b u)_i\qquad i=1,2,3 \text{ on } V.
\end{equation}

Evaluating \eqref{eq:KerPro1} at the point $\b x\in V\subset A$ and dividing by $\delta\mu(\b x)$ (which, by \eqref{eq:ThmKerA}, is non-zero) shows that
\begin{equation}
\label{eq:KerPro2}
\frac{\nabla\delta\mu(\b x)}{\delta\mu(\b x)} = \nabla\log\delta\mu(\b x)=\b a(\b x),
\end{equation}
with $\b a$ determined by \eqref{eq:KerLMu2}.

Actually, the conditions \eqref{eq:NonSingStrainT} and \eqref{eq:KerLMu2} can be used to define $\b a(\b x)$ uniquely for $\b x\in\Omega$. To see this, set $\b b_i:=\varepsilon(\b u)_i$. By \eqref{eq:NonSingStrainT},  the vectors $\b b_i$ form a basis of $\mathbb{R}^3$. Then write
\begin{equation}
\label{eq:DarstA}
	 \b a = \sum_{i=1}^3\frac{1}{\| \b b_i'\|}(\b a\cdot\b b_i')\b b_i'.
\end{equation}
Here, the vectors $\b b_i'$ are obtained by Gram-Schmidt orthogonalization, i.e., $\b b_1':=\b b_1$, and 
\begin{equation}
\label{eq:DarstB}
	\b b_i' = \b b_i - \sum_{1\leq j<i}\frac{1}{\| \b b_j'\|}(\b b_i,\b b_j')\b b_j'.
\end{equation}
The scalar products in \eqref{eq:DarstA} can be represented as
\begin{equation}
\label{eq:DarstAdotB}
	\b a\cdot\b b_i' \stackrel{\eqref{eq:DarstB}}{=} \b a\cdot\b b_i - \sum_{1\leq j<i}\frac{1}{\| \b b_j'\|}(\b b_i,\b b_j')(\b a,\b b_j').
\end{equation}

In \eqref{eq:KerLMu2}, the scalar products $\b a\cdot\b b_i$ are specified for $i=1,2,3$. Now using \eqref{eq:DarstAdotB}, we see by a simple induction argument that \eqref{eq:KerLMu2} determines $\b a\cdot \b b_i'$ in \eqref{eq:DarstA}; thus we can determine $\b a$ uniquely on $\Omega$.

Furthermore, observe that $\b u\in (H^{l+2}(\Omega))^3$ implies $\b b_i\in  (H^{l+1}(\Omega))^3$ and 
\[
	\b a\cdot\b b_i\stackrel{\eqref{eq:KerLMu2}}{=}-\nabla\cdot\b b_i\in  H^{l}(\Omega).
\]
	Using this, as well as \eqref{eq:DarstA}, \eqref{eq:DarstAdotB}, the inequality $pl> n$ and the Sobolev embedding theorem \cite[Thm.5.4.C]{Ada75}, we obtain that $\b a$ is continuous on $\Omega$.
	
Now consider the vector field $\b a$ and calculate the path integral from $\b p$ to $\b x$ to find
\[
	\int_{\b p}^{\b x}\b a(\b y)d\b y\stackrel{\eqref{eq:KerPro2}}{=}
	\int_{\b p}^{\b x} \nabla\log\delta\mu(\b y)d\b y = 
	\log\delta\mu(\b x)-\log\delta\mu(\b p)=\log\bigl(\frac{\delta\mu(\b x)}{\delta\mu(\b p)}\bigr).
\]
From this identity, we have the representation 
\begin{equation}
\label{eq:KerPro3}
	\delta\mu(\b x)=\delta\mu(\b p)\exp\int_{\b p}^{\b x}\b a(\b y)d \b y,\qquad\b x\in V
\end{equation}
for the values $\delta\mu$ on the set $V\subset A\subset \Omega$. Note that the function on the right hand side of \eqref{eq:KerPro3} is continuous and defined on the whole domain $\Omega$.

We now claim that actually, we have
\begin{equation}
\label{eq:ThmKerOmega}
V=\Omega,
\end{equation}
such that the representation formula \eqref{eq:KerPro3} holds for $\b x\in\Omega$.

Assume, on the contrary, that $V\varsubsetneq\Omega$. We then also have that 
\[
	A\varsubsetneq\Omega
\]
 (otherwise $\Omega=A=V\cup V_1$, with $V_1=A\setminus V$ open and nontrivial, $V\cap V_1=\{\}$, so $\Omega$ would not be connected).

Therefore, the assumptions of Lemma~\ref{lem:top2} are satisfied. Consequently, there exists a point $\b q\in \partial V\setminus A$, where $\partial V$ is the boundary of $V$ in $\Omega$. As $\b q\not\in A$, we have, by \eqref{eq:ThmKerA}, that 
\begin{equation}
\label{eq:ThmKerQ0}
	\delta\mu(\b q)=0.
\end{equation}
As $\b q\in\partial V$, there exists a sequence $\b v_n\in V$ with
\begin{equation}
\label{eq:ThmKerSeq}
	\b v_n\to\b q\qquad\text{in }\Omega.
\end{equation}
By the representation \eqref{eq:KerPro3}, together with \eqref{eq:ThmKerSeq}, we have 
\begin{equation}
\label{eq:ThmKerConvNon0}
	\delta\mu(\b v_n)\stackrel{\eqref{eq:KerPro3}}{=}\delta\mu(\b p)\exp\int_{\b p}^{\b v_n}\b a(\b y)d \b y\to 		\underbrace{\delta\mu(\b p)}_{\neq \;0}\underbrace{\exp\int_{\b p}^{\b q}\b a(\b y)d \b y}_{>\; 0}\neq 0
\end{equation}
On the other hand, continuity of $\delta\mu$, equation \eqref{eq:ThmKerQ0} and \eqref{eq:ThmKerSeq} imply that
\begin{equation}
\label{eq:ThmKerConv0}
\delta\mu(\b v_n)\to 0.
\end{equation}
As \eqref{eq:ThmKerConvNon0} and \eqref{eq:ThmKerConv0} contradict each other, we infer that the assumption $V\varsubsetneq\Omega$ is wrong. Therefore, as asserted in \eqref{eq:ThmKerOmega}, $V=\Omega$ holds.

Therefore, for any $(\delta\mu,\delta\b u)\in \ker \mathcal A_\mu=K_1\times\{0\}$ with $\delta\mu\neq 0$,  the representation of $\delta\mu$ in \eqref{eq:KerPro3} is valid for $\b x\in\Omega$.  This shows that \eqref{eq:KerLMu1} is a generating element for $K_1$. Therefore $\dim(K_1)= 1$.
\end{proof}

\begin{cor}
\label{cor:InjAMu2}
Let $\b u_1\neq \b u_2$ be two quasi-static elastic deformations satisfying~\eqref{eq:Modell} with different force terms $\b F_1, \b F_2$. Let the condition \eqref{eq:NonSingStrainT} hold.

As described in~\eqref{eq:SystemMu2}, let $\mathcal A_\mu^{(2)}(\delta\mu,\delta\b u_1,\delta\b u_2)$ be the corresponding linearized operator. Then we have that 
\[
	\ker (\mathcal A_\mu^{(2)})=\{(0,0,0)\}.
\]
\end{cor}
\begin{proof}
Let $(\delta\mu,\delta\b u_1,\delta\b u_2)\in \ker (\mathcal A_\mu^{(2)})$.

From \eqref{eq:SystemMu2}, we immediately get that $\delta\b u_1=\delta \b u_2=0$ on $\overline\Omega$. The other equations in \eqref{eq:SystemMu2} yield
\[
\begin{aligned}
	\nabla\cdot(\delta\mu\;\varepsilon(\b u_1))&=0\\
	\nabla\cdot(\delta\mu\;\varepsilon(\b u_2))&=0.
\end{aligned}
\]
such that together with the boundary data we have
\begin{equation}
\label{eq:DiffKerEqn}
\begin{aligned}
	\nabla\cdot(\delta\mu\;\varepsilon(\b u_1-\b u_2))&=0\\
	(\b u_1-\b u_2)|_{\partial\Omega} &=0.
	\end{aligned}	
\end{equation}
Suppose that $\delta\mu\neq 0$. Then there exists a point $\b p\in\Omega$ with $\delta\mu(\b p)\neq 0$. As in the proof of Theorem~\ref{thm:KerMu}, where we derived the representation formula \eqref{eq:KerPro3} for $\b x\in\Omega$, there exists a certain vector field $\b a$ such that
\begin{equation}
	\delta\mu(\b x)=\delta\mu(\b p)\exp\int_{\b p}^{\b x}\b a(\b y)d \b y,\qquad\b x\in \Omega.
\end{equation}
This implies that $\delta\mu(\b x)>0$ for all $\b x\in\Omega$. Therefore, the condition $\essinf_\Omega\mu=\essinf_\Omega\mu_0>~0$ in \cite[(2.2)]{MclThoYoo11} is satisfied.

The uniqueness result \cite[Thm.5.2]{MclThoYoo11} then implies that, from \eqref{eq:DiffKerEqn}, we have that $\b u_1=\b u_2$. But this is contradiction to our assumption.

Therefore, we have $\ker (\mathcal A_\mu^{(2)})=\{(0,0,0)\}$.
\end{proof}

In the subsequent part of the section we give the stability criteria for the operators $\mathcal A_\rho$, $\mathcal A_p$ and $\mathcal A_\lambda$.

\begin{theorem}
\label{thm:StabRho}
The operator $\mathcal A_\rho^{(r)}$ for $r$ measurements has a left regularizer on any smooth subdomain $W\subset\Omega\times [0,T]$ precisely when, for all $(\b x, t)\in{\overline W}$,
\begin{equation}
\label{eq:NonVanishingOfAcceleration}
	(\b u_k)_{tt}(\b x, t)\neq 0\qquad \text{for at least one measurement } 1\leq k\leq r.
\end{equation}
One has the stability estimate
\begin{equation}
\label{eq:EstARho}
	 \| \delta \rho \|_{H^{l+1}(W)} \leq \;C \sum_{k=1}^r \| \delta\b u_k \|_{H^{l+2}(W)} 	
\end{equation}
\end{theorem}
\begin{proof}
We first treat $\mathcal A_\rho$. The case of $\mathcal A_\rho^{(r)}$ follows by induction.

The stability criterion \label{eq:SolEst} in Theorem \ref{thm:Sol} is established for domains with $C^{l+\max t_j}$ boundary. Upon careful checking of the proof \cite[\S 6]{Sol73}, the only place where this assumption enters is the existence of a partition of unity. Now our domain is $\Omega\times[0,T]$, The construction of a partition of unity easily generalizes to cylindrical domains $\Omega\times[0,T]$, where $\Omega$ has $C^{l+\max t_j}$ boundary. Therefore, we can apply Theorem \ref{thm:Sol} to the problems with cylindrical domains.

The ellipticity condition has been assured in Corollary \ref{cor:ARho}, and the Lopatinskii condition is satisfied according to Proposition \ref{prop:ARho}. The assumptions in these results give the requirement $\b u_{tt}(\b x,t)_{\overline W}\neq 0$. With that, the equivalent conditions of Theorem \ref{thm:Sol} are fulfilled and we apply the result as in the proof of Theorem \ref{thm:StabMu}.

There appears no kernel in \eqref{eq:EstARho} for the following reason: The Douglis-Nirenberg numbers for the operator $\mathcal A_\rho$ are $(t_j)_{j=1}^{4}=(0,2,2,2)$ and $(s_i)_{i=1}^6=(0,0,0,-2,-2,-2)$. On the right hand side of estimate \eqref{eq:SolloEst}, only the variables with $t_j>0$ appear, which are in this case $\delta \b u_k$ for $k=1,2,3$.
\end{proof}

\begin{theorem}
\label{thm:StabP}
The operator $\mathcal A_p$ in \eqref{eq:SystemP}, considered in the stationary case, has a left regularizer on $\Omega$, and we have the estimate
\begin{equation}
\label{eq:EstAP}
	 \| \delta p \|_{H^{l+1}(\Omega)/K_3} \leq \;C \| \delta\b u \|_{H^{l+2}(\Omega)}.
\end{equation}
Here, the kernel $K_3$ consists of the (one-dimensional) space of constant functions on $\Omega$.
\end{theorem}
\begin{proof}
The proof of the stability estimate with a finite-dimensional kernel is the same as in Theorem \ref{thm:StabMu}.

Suppose that $(\delta\mu,\delta\b u)$ is in the kernel $K=\ker(\mathcal A_p)=\mathcal L_p\times \mathcal B$. Consideration of the equation 
\[
	\mathcal L_p(\delta p,\delta \b u)=\begin{pmatrix}\nabla \delta p + 2\nabla\cdot(\mu\;\varepsilon(\delta\b u))-\rho(\delta\b u)_{tt}  \\  \delta \b u\end{pmatrix}=0\qquad\text{on } \Omega
\]
shows that $\delta\b u=0$, and consequently $\nabla\delta p=0$. Therefore, we have $K=K_3\times\{0\}$, with $K_3$ the constant functions on $\Omega$.
\end{proof}

Note that, with the same method, but using Remark \ref{rem:PL}, one obtains a conditional stability result for the operator $\mathcal A_\lambda$:

\begin{theorem}
\label{thm:StabLam}
The operator $\mathcal A_\lambda^{(r)}$ corresponding to $r$ measurements, considered in the stationary case, has a left regularizer on $\Omega$ precisely when, for all $\b x\in\Omega$,
\begin{equation}
\label{eq:CondIncompressibility}
	\nabla\cdot\b u_k(\b x)\neq 0\qquad\text{for at least one measurement } 1\leq k\leq r,
\end{equation}
Then the stability estimate
\begin{equation}
\label{eq:EstALambda}
	 \| \delta \lambda \|_{H^{l+1}(\Omega)/K_4} \leq \;C \sum_{k=1}^r\| \delta\b u_k \|_{H^{l+2}(\Omega)} 	
\end{equation}
holds for a finite-dimensional kernel $K_4$.
\end{theorem}

\begin{table}
\renewcommand{\arraystretch}{3}
\begin{tabular}{l|l|l|l|l}
\hline
operator & 	$\mathcal A_\mu$ & $\mathcal A_\rho$ & $\mathcal A_p$	& $\mathcal A_\lambda$  \\
\hline
ellipticity condition &$\det(\varepsilon(\b u_k))\neq 0$	 &  $(\b u_k)_{tt}\neq 0$ & -- & $\nabla\cdot \b u_k \neq 0$   \\
\hline
\end{tabular}
\caption{Conditions for the reference state in Theorems \ref{thm:StabMu}, \ref{thm:StabRho}, \ref{thm:StabP}, \ref{thm:StabLam} to hold for every point for at least one displacement field $\b u_k$, $1\leq k\leq s$ in an imaging experiment in elastography}
\label{tab:EllCond}
\end{table}

\section{Discussion}
\label{sec:Disc}
\begin{enumerate}

\item 
The theorems show that the interior information $\b u$ provided in elastography makes the reconstruction of the biomechanical parameters $\mu$, $\rho$, as well as reconstruction of $p$, stable. We obtained criteria for the ellipticity of the linearizations $\mathcal A_p$, $\mathcal A_\mu$ and $\mathcal A_\rho$ of the quantitative elastography problems defined in Section~\ref{sec:Setting}, see table \ref{tab:EllCond}.

In the research for coupled-physics conductivity problems, ellipticity has been investigated theoretically and numerically, and found to yield optimal stability estimates, avoid blurring effects, accurate reconstruction of edges, and absence of propagation of singularities \cite{KucSte12,Kuc12,KucKun11,Bal12b,MonSte13,BalHofKnu14}.

Note that failure of ellipticity in our cases entails non-existence of a left regularizer non-existence of a left regularizer is equivalent to either $\dim\ker(\mathcal A)=\infty$ or the range of $\mathcal A$ not being closed for the particular Sobolev spaces involved \cite[XI,Thm. 2.3]{Con90}. This does not mean that necessarily, the linearized problem will be unstable for all data in any function space. For example, consider the case of Corollary \ref{cor:AM}: at a point $\b x$, there might be just one direction $\gb \xi$ for which ellipticity does not hold. Then one can form the conjecture that reconstruction can still be stable if there is no edge along this direction (see the related discussion in \cite[6(ii)]{KucSte12}). We plan to address this in future work.

\item The ellipticity conditions for $\mathcal A_\lambda$ and $\mathcal A_\rho$ seem to be natural. Concerning $\lambda$, literature actually often assumes the incompressibility condition $\nabla\cdot\b u=0$ on the whole of $\Omega$ \cite{RagYag94,AmmGarJou10,ManOliDreMahKru01,BarGok04}. In this case, of course, the measurement data are not dependent on $\lambda$, so this parameter cannot be reconstructed then. -- But in the compressible case, where $\nabla\cdot\b u\neq 0$ on the whole of $\Omega$, there still might be single points $\b x$ at which $\nabla\cdot\b u(\b x)=0$. Notice that, as stated in Remark~\ref{rem:PL}, the ellipticity analysis along the lines of this article then entails that at such points $\b x$, ellipticity is lost and every direction is a characteristic. -- Concerning the particular data one has, it might then be better to reconstruct the pressure $p=\lambda\nabla\cdot\b u$, with the operator $\mathcal A_p$ being always elliptic.

Similarly for $\rho$: If $\b u_{tt}=0$ on the whole of $\Omega$, the parameter $\rho$ does not appear in the model, so it cannot be reconstructed from the measurements. If, on the other hand, $\b u_{tt}(\b x)=0$ only for particular points $\b x$, the analysis says that ellipticity is lost at these points $\b x$, and every direction is a characteristic for $\mathcal A_\rho$ there.

\item The ellipticity condition for reconstruction of $\mu$ turned out to be the non-singular strain condition in \eqref{eq:NonSingStrainT}, which is a generalization of the condition in \eqref{eq:BalInvert}. Apart from this characterization, points of singular strain have been found in experiments, namely at the intersection of nodal lines or surfaces in early experiments of elastography using eigenmodes (see \cite{ParTayGraRub05,ParLer92,TayRubPar00}). Empirically, it was observed that these patterns could be avoided by choosing multi-frequency excitation functions $\b F$ \cite{ParTayGraRub05}.
\end{enumerate}

\section{Conclusion}
We have applied a general method of linear PDE to linearized problems in quantitative elastography in $\mathbb{R}^3$, with interior data given. We analyzed ellipticity conditions of the PDE problem augmented with the interior data. We deduced simple criteria for the stability of the linearization. This analysis revealed stable reconstruction of the shear modulus $\mu$ and the hydrostatic pressure $p=\lambda\nabla\cdot\b u$, but pointed to a difficulty of reconstruction of $\lambda$. For the reconstruction of $\mu$ and $\rho$, the kernel in the linearization was shown to be trivial for choice of two measurements. The results give a mathematical explanation which biomechanical parameters can be stably reconstructed from interior measurement data $\b u$.

\subsection*{Acknowledgements}
We thank Joyce McLaughlin, Dustin Steinhauer, Guillaume Bal, Josef Schicho, Jos\'e Iglesias Martinez and Kristoffer Hoffmann for helpful discussions and acknowledge support from the Austrian Science Fund (FWF) in project S10505-N20.

\appendix
\section{Appendix}
We give here the proof of two topological lemmas which we use in the determination of the kernel in Theorem~\ref{thm:KerMu}

\begin{lem}
\label{lem:top1}
Let $A\subset\Omega$ be open, and let $\b p\in A$. Let $V$ be the connected component of  $\b p$ in the topology of $A\subset\Omega$.  Then $V$ is open in $\Omega$.
\end{lem}
\begin{proof}
Let $\b x\in V\subset A$ be an arbitrary point in $V$. As $\b x\in A$ and $A$ is open, there exists an $\varepsilon > 0$ such that 
\[
	U_1:=\{ \b z \in\Omega: | \b x-\b z|<\varepsilon\}\subset A.
\]
Observe that the set $U_1$ is connected and $\b x\in V\cap U_1$. From \cite[Thm.23.3]{Mun00}, it then follows that $V\cup U_1\subset A$ is a connected set.

Among all subsets of $A$ which are connected and contain $\b p$, the component $V$ is maximal. Therefore $\b p\in V\cup U_1=V$, or equivalently $U_1\subset V$. This shows that $V$ is open in $\Omega$.
\end{proof}

\begin{lem}
\label{lem:top2}
Let $A\varsubsetneq\Omega\subset\mathbb{R}^n$ be open and bounded, and let $\b p\in A$. Let $V$ be the connected component of  $\b p$ in the topology of $A\subset\Omega$. Let $\partial V$ be the boundary of $V$ in the topology of $\Omega$. 
Then there exists a point 
\[
	\b q\in\partial V\setminus\ A.
\]
\end{lem}
\begin{proof}
We use Lemma~\ref{lem:top1} and prove the statement in two steps: first, we find a point $\b q\in\partial V\setminus V$; second, we show that $\b q\not\in A$.\\

{\bf Claim 1:} There exists a point $\b q\in\partial V\setminus V$.\\
We have that $V\subset A\varsubsetneq \Omega$. Therefore, there exists an element 
\begin{equation}
\label{eq:lemDefz}
	\b y\in\Omega\setminus V.
\end{equation}
Consider the mapping
\[
	\begin{aligned}
		f: \;&\overline V\to \mathbb{R} \\
		&  \b v\mapsto |\b y-\b v|.
	\end{aligned}
\]
Observe that $\overline V\subset\Omega$ is closed and bounded, hence a compact set; observe also that $f$ is continuous. Therefore, a minimum exists, that is:
\begin{equation}
\label{eq:lemMin}
	\exists \;\b q\in \overline V: |\b y - \b q| = \min_{\b v\in\overline V} \{|\b y- \b v|\}.
\end{equation}
We now show that, actually, the point $\b q\in \overline V=V\cup\partial V$ is not contained in $V$. Once this is shown, Claim 1 is proven.

Assume, on the contrary, that $\b q\in V$. According to Lemma~\ref{lem:top1}, we then would have an $\varepsilon$, such that
\[
	U_2:= \{ \b z:| \b q-\b z|<\varepsilon\} \subset V.
\]
Without loss of generality, we can assume $\varepsilon<2$. Now, using the element $\b y$ from \eqref{eq:lemDefz}, define the point
\[
	\b w:= \b q + \frac{\varepsilon}{2}(\b y-\b q),\qquad \b w\in  U_2.
\]
Then calculate
\[
\begin{aligned}
	| \b y-\b w| &= | \b y-\b q - \frac{\varepsilon}{2}(\b y-\b q) | = | (\b y-\b q)(1-\frac{\varepsilon}{2})|
\\
& \leq | \b y-\b q| \underbrace{(1-\frac{\varepsilon}{2})}_{<\;1}<| \b y-\b q|.
\end{aligned}
\]
This would contradict \eqref{eq:lemMin}. -- Therefore, $\b q\not\in V$.\\

{\bf Claim 2:} The point $\b q$ in \eqref{eq:lemMin} does not belong to $A$.\\
We prove this claim indirectly. Assume that 
\begin{equation}
\label{eq:lemQinA}
	\b q\in A.
\end{equation}
Recall that, according to Claim 1, $\b q\in\partial V$, where $\partial V$ is the boundary of $V$ in~$\Omega$.
Hence there exists a sequence $\b v_n\in V$ with $\b v_n\to\b q$ in the topology of $\Omega$.

We assert that 
\begin{equation}
\label{eq:lemConvA}
		\b v_n\to \b q\qquad\text{in the topology of } A.	
\end{equation}
To see this, choose an open set $U_3\subset A$ with $\b q\in U_3$. Because $A$ is open in $\Omega$, $U_3$ is open in $\Omega$ as well. Now the elements $\b v_n$ converge to $\b q$ in $\Omega$; therefore, there exists an $N$, such that for all $n\geq N: \b v_n\in U_3$; hence we have \eqref{eq:lemConvA}.

The set $V$, which is the connected component of the point $\b p$, is closed in the topology of $A$ \cite[Thm.23.4]{Mun00}. But a closed set contains all its limit points. Therefore, with~\eqref{eq:lemConvA}, we would have that the limit of the sequence $\b v_n\in V$ lies in $V$, so $\b q\in V$. But this is a contradiction to Claim 1. -- Therefore, contrary to \eqref{eq:lemQinA}, we have $\b q\not\in A$.

\end{proof}

\small

\bibliographystyle{plain}

\def\cprime{$'$}
  \providecommand{\noopsort}[1]{}\def\ocirc#1{\ifmmode\setbox0=\hbox{$#1$}\dimen0=\ht0
  \advance\dimen0 by1pt\rlap{\hbox to\wd0{\hss\raise\dimen0
  \hbox{\hskip.2em$\scriptscriptstyle\circ$}\hss}}#1\else {\accent"17 #1}\fi}
  \def\cprime{$'$}

\end{document}